\documentclass[11pt, final]{article}
\usepackage{packages}

\title{\bf{Interior regularity of area minimizing currents within a $C^{2,\alpha}$-submanifold}}
\author{
Stefano Nardulli\thanks{Universidade Federal do ABC - stefano.nardulli@ufabc.edu.br} \ and Reinaldo Resende\thanks{Carnegie Mellon University - rresende@andrew.cmu.edu}}
\vspace{-6cm}

\begin{document}

\date{}

\maketitle

\begin{abstract}
Given an area-minimizing integral $m$-current in $\Sigma$, we prove that the Hausdorff dimension of the interior singular set of $T$ cannot exceed $m-2$, provided that $\Sigma$ is an embedded $(m+\bar{n})$-submanifold of $\R^{m+n}$ of class $C^{2,\alpha}$, where $\alpha>0$. This result establishes the complete counterpart, in the arbitrary codimension setting, of the interior regularity theory for area-minimizing integral hypercurrents within a Riemannian manifold of class $C^{2,\alpha}$.
\end{abstract}

\setcounter{tocdepth}{2} 
\tableofcontents

\newpage

\section{Introduction}  

Let $m \geq 2$, $n \geq 1$, $\bar{n} \geq 0$, $\alpha > 0$, and $\Sigma \subset \R^{m+n}$ be a $C^{2,\alpha}$-submanifold of dimension $m + \bar{n}$. In order to state our main result, we need to introduce the notion of regular interior points in the context of integral currents. This notion asserts that the support of $T$ is equal to a $C^{2,\alpha}$-submanifold of $\Sigma$ around points that are far from the boundary. We refer the reader to \Cref{S:Preliminaries} {for} notation and basic definitions.

\begin{defi*}\label{defi:InteriorRegSingPoints}
Let $T$ be an $m$-dimensional integral current in $\Sigma$. 
We say that $p\in\operatorname{spt}(T)\setminus\operatorname{spt}(\partial T)$ is an {\emph{interior regular point of $T$}}, if there exists a neighborhood $U \ni p$ and an $m$-dimensional $C^{2,\alpha}$-submanifold $\mathcal{S} \subset \Sigma$ such that $\operatorname{spt}(T) \cap U = \mathcal{S} \cap U$. The set of such points will be denoted by $\operatorname{Reg}_{\mathrm{i}}(T)$. We call $\operatorname{Sing}_{\mathrm{i}}(T) := \operatorname{spt}(T) \setminus (\operatorname{Reg}_{\mathrm{i}}(T)\cup \operatorname{spt}(\partial T))$ the set of {\emph{interior singular points of $T$}}.
\end{defi*}

We now define the notion of area minimality that we will consider throughout this note.

\begin{defi*}\label{defi:AreaMinimizingCurrents} Let $U\subset\R^{m+n}$ be an open set, we say that an integral $m$-current $T$ is {\emph{area minimizing in $\Sigma\cap U$}} where $\Sigma$ is an embedded $C^{2,\alpha}$ {submanifold} of $\R^{m+n}$, if 
\begin{equation*}
    \|T\|(U)\le \|T+\partial S\|(U),
\end{equation*}
for every integral $(m+1)$-current $S$ with $\operatorname{spt}(S)\subseteq\Sigma\cap U$.
\end{defi*}

Given the definitions above, we are ready to state the main result of this paper.

\begin{thm}\label{T:MAIN}
Let $\alpha>0$, $m \geq 2$, $n \geq 1$, $\bar{n} \geq 0$, and let $\Sigma$ be an embedded $(m+\bar{n})$-submanifold of class $C^{2,\alpha}$ of $\R^{m+n}$. If $T$ {is} an area minimizing integral $m$-current in $\Sigma$, then 
\begin{equation*}
    \mathrm{dim}_{\cH}(\operatorname{Sing}_{\mathrm{i}}(T))\le m-2.
\end{equation*}
\end{thm}

This result provides a complete counterpart in the context of arbitrary codimension for the interior regularity theory concerning area-minimizing integral currents in codimension $1$. Specifically, through an application of Nash's {isometric} embedding theorem (for a comprehensive exposition of Nash's {isometric} embedding theorem, see \cite[Section 4.1]{de2019masterpieces}; for Nash's original proof assuming at least $C^3$ regularity, see \cite{nash1956imbedding}; and for a generalization assuming $C^{k,\beta}$-regularity with $k+\beta > 2$, see \cite{jacobowitz1972implicit}), \Cref{T:MAIN} implies the following.

\begin{thm}\label{T:RiemManifolds}
Let $\alpha>0$, $m \geq 2$, $\bar{n} {>} 0$, and let $(\overline{\Sigma},g)$ be a Riemannian $(m+\bar{n})$-manifold of class $C^{2,\alpha}$. If $T$ {is} an area minimizing integral $m$-current in $\overline{\Sigma}$, then 
\begin{equation*}
    \mathrm{dim}_{\cH_g}(\operatorname{Sing}_{\mathrm{i}}(T))\le m-2.
\end{equation*}
\end{thm}
\begin{rem}
Note that the assumption on \Cref{T:MAIN} is weaker than what we ask for in \Cref{T:RiemManifolds}. Indeed, if we consider $\Sigma$ already embedded in $\R^{m+n}$ and of class $C^{2,\alpha}$ as in \Cref{T:MAIN}, then its Riemannian structure (i.e., the metric tensor) is of class $C^{1,\alpha}$ which is not covered by \Cref{T:RiemManifolds}.
\end{rem}

Nash's {isometric} embedding theorem remains a wide open problem for $C^2$ Riemannian manifolds. Specifically, the challenge lies in ensuring that the embedded Riemannian structure inherits $C^2$ regularity. While it is known that one can achieve $C^{1,\alpha}$-regularity for the {isometrically} embedded Riemannian structure, the question of attaining $C^2$ regularity remains unanswered. With that being said, one would need to approach \Cref{T:RiemManifolds} intrinsically in order to relax the assumption from $C^{2,\alpha}$ to $C^2$. Although the authors expect such a result to be true, rerunning all the machinery presented in \cite{DS2,DS3,DS4,DS5} intrinsically seems quite challenging. Indeed, in the case that the density is $1$, an intrinsic proof of Allard's $\varepsilon$-regularity theorem in $C^2$ Riemannian manifolds will be provided in \cite{stefanomarcos}.

It is well-known that the dimensional bound in \Cref{T:MAIN} is optimal. Indeed, consider the $2$-dimensional integral current $T$ in $\R^4$ induced by ${F:=}\{(z,w)\in \mathbb{C}^2: z^2 = w^3\}$ (which, being a holomorphic subvariety of $\mathbb{C}^2$, is area-minimizing by a famous theorem due to Federer), this current $T$ clearly has $0$ as a singular point, demonstrating the optimality of the dimensional bound in \Cref{T:MAIN}.

\Cref{T:MAIN} was first proved by Almgren in his celebrated work \cite{Alm}, considering $\Sigma$ an already embedded submanifold of some higher dimensional Euclidean space and of class $C^5$. While Almgren's approach was innovative and groundbreaking, it is also intricate and lengthy. In their series of papers \cite{DS1,DS2,DS3,DS4,DS5}, De Lellis and Spadaro introduced new techniques to tackle the regularity theory developed by Almgren. They provided a much simpler and shorter proof of \Cref{T:MAIN}, in which they also achieved an improvement by weakening the assumption on the regularity of $\Sigma$ from $C^5$ to $C^{3,\alpha}$. In this work, we further improve upon their results by extending them to the case where $\Sigma$ is merely $C^{2,\alpha}$.

\subsection*{General framework}

We will follow the approach introduced in \cite{DS1,DS2,DS3,DS4,DS5} to prove \Cref{T:MAIN}. {In their approach, as in \cite{Alm}, the goal is to assume, by way of contradiction, that the area minimizing current $T$ has too large of a singular set and then approximate $T$ with the graph of Lipschitz functions that are ``approximately minimizers for the Dirichlet energy". Therefore, a contradiction argument can be carried out through blow-ups in order to contradict classical regularity theory for Dirichlet minimizers.}

{Although this approach is rather classical in regularity theory for geometric variational problems, in this setting, one highly nontrivial problem immediately arises: we cannot hope to suitably approximate $T$ with the graph of $\mathbb{R}^n$-valued functions due to branching behaviors, see Federer's variety example $F$ above. One of Almgren's groundbreaking ideas was the introduction of the concept of ``multi-valued functions" (or $Q$-valued functions, which are functions with values in the metric space $\mathcal{A}_Q(\mathbb{R}^n)$ given by all the sums of $Q$ Dirac masses of vectors in $\mathbb{R}^n$), written as $f:\Omega\subset \mathbb{R}^m\to \mathcal{A}_Q(\mathbb{R}^n)$, which captures this ``multi-sheeted" behavior of currents like $F$. The regularity properties of $Q$-valued functions that minimize a suitable notion of Dirichlet energy are significantly weaker than those of $\mathbb{R}^n$-valued harmonic functions. In fact, the optimal result states that their singular set has Hausdorff dimension at most $m-2$. This is investigated in \cite{Alm,DS1}.}

{Still at the level of the linear problem, there is a subtlety: for a multi-valued Dirichlet minimizer $f$, taking blow-ups of $f$ does not necessarily preserve singularities in the limit. This adds a significant layer of complexity because, if singularities are lost in the blow-up process, we lose the very structure we aim to analyze and contradict within the argument. This difficulty is overcome by subtracting an ``average" (denote by $\etaa\circ f$) over all the sheets of $f$ and leveraging the almost monotonicity of the so-called frequency function of $f$, see \cite{Alm,DS1}.}

{The space of $Q$-valued functions provides a natural framework, and a regularity theory, of functions whose graphs (precisely, the current induced by its graph, see \cite{DS2}) approximate currents exhibiting branching behavior. This is shown in \cite{DS3}, where the authors exploit the minimality of $T$ to show that such approximations can be chosen as ``approximately minimizers of the Dirichlet energy" while satisfying suitable estimates.} 

{For the nonlinear problem, it is also needed to remove a sort of average of the sheets of $T$ that, in a rough sense, captures its ``regular part", in order to ``preserve" singularities in the limit. This process is highly intricate and provides as an average for $T$ a submanifold $\cM$ in $\Sigma$, known as the \emph{center manifold}, whose construction is detailed in \cite{DS4}. The construction of $\cM$ relies on a refined Whitney decomposition, which crucially depends on the estimates for the multi-valued ``approximately Dirichlet minimizing" approximations of $T$ and the fact that $T$ is an area minimizer. Once the center manifold $\cM$ is established, we can ``subtract" $\cM$ from $T$ and construct the ``average-free" approximations $N$ of $T$. These take the form of ``approximately Dirichlet minimizing" multi-valued functions $N$, which are now defined on $\cM$ and whose graphs serve as approximations of $T$, see \cite{DS4}.}

{To carry out the blow-up argument, the intervals of flattening are introduced, which captures, among other properties, the radii at which $T$ exhibits almost quadratic excess-decay estimates. For each rescaling of $N$ by these radii, there is an associated frequency function suitably defined. A crucial step in the analysis is proving an almost-monotonicity formula for this frequency function where the minimality of $T$ in $\Sigma$ plays a big role. By leveraging this formula, along with all the preceding constructions and the assumption that the singular set of $T$ has Hausdorff dimension strictly greater than $m-2$, we establish a contradiction via these suitable blow-ups of $N$. Specifically, using the minimality of $T$ in $\Sigma$, we obtain at the limit of these blow-ups of $N$ a multi-valued Dirichlet minimizer whose singular set exceeds the Hausdorff dimension bound of $m-2$. This contradicts the regularity theory for the linear problem, thereby completing the argument. This can be found in \cite{DS5}.}

\subsection*{What is new in this work}

The techniques presented in \cite{DS1,DS3} do not depend on the $C^{3,\alpha}$-regularity of $\Sigma$. In \cite{DS3}, the authors do require $\Sigma$ to be at least $C^2$, which is then applicable to the setting considered in the present work.

In contrast, in \cite{DS2,DS4,DS5}, the authors utilize the full strength of the $C^{3,\alpha}$-regularity assumption. Specifically, in \cite{DS4}, they need to construct the so-called center manifold $\cM\subset\Sigma$, which enjoys $C^{3,\alpha}$-regularity due to $\Sigma\in C^{3,\alpha}$, to subsequently leverage this regularity in controlling certain error terms arising in \cite{DS5} after some applications of \cite{DS2}. The control over these error terms, with very specific decay exponents, is delicate and crucial for proceeding with proving the so-called almost-monotonicity of the frequency function and the blow-up argument in \cite{DS5}. Therefore, the question of whether the $C^{3,\alpha}$-regularity assumption could be eliminated has remained open.

In our approach, we weaken the $C^{3,\alpha}$-regularity assumption by constructing what we term the external center manifold $\cM^\ast$, potentially lying outside $\Sigma$ (contrary to $\cM$ which is always contained in $\Sigma$), which culminates in several complexities when approaching \Cref{T:MAIN}. Nonetheless, our method of proof can still establish that $\cM^\ast$ is $C^{3,\alpha}$ (\Cref{T:ECM}) using only the elliptic system induced by the $C^{2,\alpha}$-regularity of $\Sigma$ and the minimality of $T$ in $\Sigma$. 

Another delicate issue in this work is the construction of the so-called $\cM$-normal approximations $N$, as introduced in \cite{DS4}, which serve as an average-free approximation of $T$. These approximations, termed $\cM^\ast$-normal approximations and denoted by $N^\ast$, are defined on $\cM^\ast$ in our setting. This enables us to utilize all the results in \cite{DS2}, as they solely rely on the $C^{3,\alpha}$-regularity of the domain of $N^\ast$. However, it is crucial that $N^\ast$ ranges within $\Sigma$, as we aim to utilize the minimality of $T$ in $\Sigma$ to compare it with the first variation of the multi-valued graph of $N^\ast$, thereby obtaining precise error terms that lead to the almost-monotonicity of the frequency function. In Section \ref{S:NA}, we provide a construction to ensure that the image of $N^\ast$ is trapped within $\Sigma$.

The geometric constructions and the proof of the almost-monotonicity of the frequency function in \cite{DS5} heavily rely also on the fact that $\cM \subset \Sigma$, which is not necessarily true for $\cM^\ast$. Hence, we provide a less geometric approach (compared to \cite{DS5}) in \Cref{S:BU} to carry out a blow-up argument that allow us to demonstrate \Cref{T:MAIN}.

\subsection*{Acknowledgements}

Both authors thank C. De Lellis for useful discussions and for reading a previous version of this note and L. Eduardo Osorio Acevedo for drawing the pictures included in this note. We also thank the anonymous referee for suggestions that considerably improved this manuscript.

R. Resende thanks UFABC (Universidade Federal do ABC) for their hospitality during a research visit funded by CNPq-Research Project: 441922/2023-6. S. Nardulli was supported by the research grants ``Aux\'ilio a Jovem Pesquisador em Centros Emergentes" FAPESP: 2021/05256-0 and ``Bolsa de Produtividade em Pesquisa 1D" CNPq: 312327/2021-8.

\section{Preliminaries and Whitney decomposition}\label{S:Preliminaries}

We will use the following notations: $\ball{p}{r} := \{x\in\R^{m+n}: |x-p|<r\}$, $\pi_0 := \R^m\times\{0\} \subset\R^{m+n}$. For any {affine} $m$-plane $\pi$ and $x\in\pi$, we define $\tball{x}{r}{\pi} := \ball{x}{r}\cap\pi$ and $\cyltilted{x}{r}{\pi} = \tball{x}{r}{\pi}\times \pi^\perp$, when $\pi = \pi_0$ we simply omit it from the notation. We also fix $\cH^s$ for the $s$-dimensional Hausdorff measure in $\R^{m+n}$, $\omega_m := \cH^m(\mathrm{B}_{1}(0))$, and $\mathrm{dim}_\cH(S)$ to be the Hausdorff dimension of a set $S$. 

{Given any vector subspace $\pi\subset\R^{m+n}$, we denote the orthogonal projection onto $\pi$ by $\bp_{\pi}$ and $\bp:= \bp_{\pi_0}$.}

For basics on $Q$-valued functions, we refer the reader to \cite{DS1,DS2}. We refer the reader to \cite{Fed} for classical theory of currents. We will call any element $T$ of $\bI_m\left(U\right)$ an integral $m$-current in $U\subset \R^{m+n}$ and $\|T\|$ denotes its total variation measure. As in \cite{DS4}, let $\bE(T,\ball{p}{r}, \pi)$ to be the spherical excess of $T$ in $\ball{p}{r}$ with respect to $\pi$ { and $\bh\left(T, B, \pi\right)$ be the height of $T$ with respect to $\pi$ in $B$ for any $B\subset\R^{m+n}$. If $\pi = \pi_0$, we drop it from the notation.}


\begin{assump}\label{A:CM}
Let \(T\) be an \(m\)-dimensional integral current of \(\R^{m+n}\) with support in \(\Sigma\cap\oball{6\sqrt{m}}\), we assume that
\begin{align}
    \Sigma\text{ is a }C^{2,\alpha}\text{-submanifold of }&\oball{7\sqrt{m}},\\
    T \text{ is an area minimizing current in } \Sigma\cap&\oball{6\sqrt{m}}, \\
    \partial T\res\oball{6\sqrt{m}} &= 0, \label{A:no_bdr} \\
    \Theta^m(T,0) = Q \in\N&\setminus\{0\}, \label{A:dens=Q} \\
    \frac{\|T\|(\oball{6\sqrt{m}r})}{r^m} - Q \omega_m\left(6\sqrt{m}\right)^m \leq \varepsilon_{cm}^2, \quad \forall r&\leq 1, \label{A:bound_dens} \\
    \bm_0 = \max\{ \bE, c(\Sigma)\} \leq \varepsilon_{cm}^2 &\ll 1,
    \label{A:DiniCondition}
\end{align}
where \(\varepsilon_{cm}\) is a small positive number to be specified later and $\bE:=\bE(T,\oball{6\sqrt{m}})=\bE(T,\oball{6\sqrt{m}},\pi_0)$. {We also assume that, for each $p\in\Sigma$, the submanifold $\Sigma$ is given by the graph of a $C^{2,\alpha}$-map $\Psi_p:T_p\Sigma\cap\oball{7\sqrt{m}}\to T_p\Sigma^{\perp}$} and $c(\Sigma) := \sup_{p\in\Sigma\cap\oball{6\sqrt{m}}}\|D\Psi_p\|_{C^{1,\alpha}}$.
\end{assump}
\begin{rem}\label{A:nearest pt proj}
{Notice that, it is well-known that if $c(\Sigma)$ is smaller than a geometric constant, then the nearest point projection onto $\Sigma$ is well-defined and $C^{1,\alpha}$ in $\oball{6\sqrt{m}}$. Hence, up to further decreasing $\varepsilon_{cm}$, we can assume that we are always working within the domain of the nearest point projection. We refer the reader to \cite[Theorem 3.4]{AamariKimEtAl2019} for an estimate of the normal injectivity radius or the reach of a submanifold of the Euclidean space.}
\end{rem}
\begin{rem}
{We will denote by $\Psi:\R^{m+\bar{n}}\to \R^{n-\bar{n}}$ the $C^{2,\alpha}$-map provided by \cite[Lemma 1.5]{DS4} whose graph we can assume (without loss of generality) to be $\Sigma$ and also 
\begin{equation}\label{Sigma-close-to-pi_0}
    |\R^{m+\bar{n}}\times\{0\} - T_0\Sigma| \leq C\bm_0^{1/2}.
\end{equation}
We mention that \cite[Lemma 1.5]{DS4} is stated for $C^{3,\alpha}$-submanifolds; however, it is straightforward to see that their proof provides the same result in our context.}
\end{rem}

We will say that a constant is a \emph{dimensional constant} when it depends only on \(m\) and \(n\), and \emph{geometric constant} when it depends on \(m,n\) and \(Q\). We always use the notations $C_0$ and $c_0$ for large and small geometric constants, respectively. We recall that from \cite[Lemma 1.6]{DS4}, we have as a consequence of Assumption \ref{A:CM}, that 
\begin{equation}\label{A:projectionQ}
    \left(\bp_{\#} T\right)\res\oball{11\sqrt{m}/2} = Q\a{{\obaseball{11\sqrt{m}/2}}},
\end{equation}
and the height bound estimate that will be helpful to derive more refined bounds. Precisely, there exists positive geometric constants \(C_0, c_0>0\) such that 
\begin{equation}\label{L:1HB}
    \bh\left(T\res\oball{\frac{23\sqrt{m}}{4}}, \ocyl{5\sqrt{m}}\right) \leq C_0\bm_0^{\frac{1}{2m}}, \text{ whenever } \varepsilon_{cm} \leq c_0.
\end{equation}

\subsection{Whitney decomposition, behaviour of the various parameters, and stopping conditions}\label{subsec:Whitney}

\noindent We begin by setting up the Whitney decomposition that will allow us {to} construct the external center manifold. To define this Whitney decomposition, we define the so-called stopping conditions in which we use precise rates of decay for the excess and height of the current $T$ to determine whether or not a $m$-cube should belong to such Whitney decomposition. 

We will fix some notations on $m$-cubes and Whitney decomposition that will be used in this work. For each \(j \in \N\), we denote by \(\sC^j\) the family of closed dyadic cubes \(L\) of \(\pi_0:=\R^{m}\times\{0\}\cong \R^m\) (such identification will be used without further mention) of the form
\begin{equation*}
    \left[a_1, a_1+2 \ell\right] \times \cdots \times\left[a_m, a_m+2 \ell\right] \subset \pi_0,
\end{equation*}
where \(\ell(L):= \ell\) and \(2^{-j} = \ell(L)\) is half of the side-length of the cube, \(a_i \in 2^{1-j} \mathbb{Z}\), and we also set 
\begin{equation}\label{E:size_center_of_cubes}
    -4 \leq a_i \leq a_i+2 \ell(L) \leq 4.
\end{equation} 

We fix the notation for the center \(x_L\) of the cube \(L\), i.e. \(x_L:=\left(a_1+\ell, \ldots, a_m+\ell\right)\). Next we set \(\sC:=\bigcup_{j \in \N} \sC^j\). 

\begin{defi}\label{D:family}
If \(H, L \in \sC\), we say that:
\begin{enumerate}[\upshape (a)]
    \item \(H\) is a \emph{descendant} of \(L\) and \(L\) is an \emph{ancestor} of \(H\), if \(H\subset L\);

    \item \(H\) is a \emph{child} of \(L\) and \(L\) is the \emph{parent} of \(H\), if \(H\subset L\) and \(\ell (H) = \frac{1}{2} \ell (L)\);

    \item \(H\) and \(L\) are \emph{neighbors} if \(\frac{1}{2} \ell (L) \leq \ell (H) \leq 2\ell (L)\) and \(H\cap L \neq \emptyset\).;

    \item \(\{H_i\}_{i=N_0}^{i_0}\) is called \emph{the genealogical tree of \(H\)}, if \(H_{i_0}:=H\), \(H_{N_0}\in\sC^{N_0}\), and each \(H_i\) is the parent of \(H_{i+1}\).
\end{enumerate}
\end{defi}

Let us recall the classical well known notion of Whitney decomposition in which we decompose \([-4,4]^m\) {satisfying} nice interactions among the cubes.

\begin{defi}\label{D:WhitneyDecomposition} 
A \emph{Whitney decomposition of \([-4,4]^m \subset \pi_0\)} consists of a closed set \(\bGamma \subset [-4,4]^m\) and a family \(\sW \subset \sC\) satisfying the following properties:
\begin{enumerate}[\upshape (w1)]
    \item \(\bGamma \cup \bigcup_{L \in \sW} L=[-4,4]^m\), and \(\bGamma\) does not intersect any element of \(\sW\);

    \item the interiors of any pair of distinct cubes \(H,L \in \sW\) are disjoint, i.e., \(\mathring{H}\cap\mathring{L}=\emptyset\);

    \item\label{w3} if \(H, L \in \sW\) have nonempty intersection, then \(\frac{1}{2} \ell\left(H\right) \leq \ell\left(L\right) \leq\) \(2 \ell\left(H\right)\).
\end{enumerate}
\end{defi}

Note that, if \(H,L\in\sW\) have nonempty intersection and \((\bGamma,\sW)\) is a Whitney decomposition, then \(H\) and \(L\) are neighbors. We now show the behaviour of the various parameters that we need to prove the existence of the external center manifold. This hierarchy of the parameters ensures that we can indeed choose such parameters in a way that all the assumptions of the statements of \cref{S:CM} are satisfies. Henceforth, we denote by \(\varepsilon_{la}>0\) and \(\gamma_{la}>0\) the geometric constants \(\varepsilon_1\) and \(\gamma_1\) given by the strong Lipschitz approximation, \cite[Thm 2.4]{DS3}. Note that the author in \cite{DS3} only assume $C^2$ regularity of $\Sigma$. 

\begin{assump}[Hierarchy of the parameters for the external center manifold]\label{A:CM:parameters}
$\gamma_h$ and $\gamma_e$ are two fixed exponents satisfying:
\begin{equation*}
    \gamma_h = 4\gamma_e = \min\left\{\frac{1}{2m},\frac{\gamma_{la}}{100}\right\}.
\end{equation*}
$M_0$ is positive real number and $N_0 = N_0(M_0)$ a natural number under the following assumptions:
\begin{equation*}
    M_0 \geq C_0(m,n,\bar{n},Q) \geq 4, \text{ and } \sqrt{m}M_02^{7-N_0}\leq 1.
\end{equation*}
$C_e$ and $C_h$ are positive real numbers for which, throughout the paper, we always assume:
\begin{itemize}
    \item \(C_e = C_e(\gamma_e,M_0,N_0) \geq 6^mM_0^{-m}2^{(N_0+6)(m+2-2\gamma_e)}\),

    \item \(C_h = C_h(\gamma_h,M_0,N_0,C_e(\gamma_e,M_0,N_0)) \geq C_e^{1/2}C_0\left(M_0 + 2^{N_0(1+\gamma_h)}\right)\).
\end{itemize}
Finally, \(\varepsilon_{cm}=\varepsilon_{cm}(\gamma_e,\gamma_h,M_0,N_0,C_e,C_h)>0\) will be a small parameter.
\end{assump}

This hierarchy is the very same that appears in \cite[Assump 1.8 and 1.9]{DS4}. From now on, we give several properties of our Whitney decomposition given by the stopping cubes conditions. First of all, thanks to property \eqref{A:projectionQ}, given a cube \(H\in\sC\), there exists \(y_L\in\R^n\) such that \(p_L:=(x_L,y_L)\in\spt(T)\). We also define
\begin{equation*}
    r_L:= M_0\sqrt{m}\ell(L).
\end{equation*}
We will usually call \emph{generation of the cube \(L\)} the number \(-\log_2(\ell(L))\in\N \) and \emph{satellite balls} the balls \(\sball{L}\).

\begin{defi}[Stopping cubes conditions]\label{D:stopping_condition}
For \(L\in\sC\), we define the families of cubes \(\sS \subset \sC\) and \(\sW=\sW_e \cup \sW_h \cup \sW_n \subset \sC\) with the convention that 
\begin{equation*}\sS^j=\sS \cap \sC^j, \quad \sW^j=\sW \cap \sC^j, \quad\text{and}\quad \sW_{\square}^j=\sW_{\square} \cap \sC^j, \quad\text{for}\quad \square=h, n, e.\end{equation*}
We set \(\sW^i=\sS^i=\emptyset\) for \(i<N_0\). We define the next generations, i.e., \(j \geq N_0\), inductively: if no ancestor of \(L \in \sC^j\) is in \(\sW\), then
\begin{align}
     &L \in \sW_e^j \text{ if } \mathbf{E}\left(T, \sball{L}\right)>C_e \bm_0\ell(L)^{2-2 \gamma_e};\eqname{(Stops for the excess)}\\
     &L \in \sW_h^j \text{ if } L \notin \sW_e^j \text{ and } \bh\left(T, \sball{L}\right)>C_h \bm_0^{\frac{1}{2m}} \ell(L)^{1+\gamma_h};\eqname{(Stops for the height)}\\
     &L \in \sW_n^j \text{ if } L \notin \sW_e^j \cup \sW_h^j \text{ but it intersects an element of } \sW^{j-1};\eqname{(Stops by neighboring)}\\
     &L \in \sS^j \text{ if } L \notin \sW_e^j \cup \sW_h^j\cup \sW^j_n .\eqname{(Subdividing cubes)}
\end{align}
Lastly, we define what we call the \emph{contact set} as follows
\begin{equation*}
\bGamma:=[-4,4]^m \backslash \bigcup_{L \in \sW} L=\bigcap_{j \geq N_0} \bigcup_{L \in \sS^j} L .
\end{equation*}
\end{defi}

Observe that, if \(j>N_0\) and \(L \in \sS^j \cup \sW^j\), then necessarily its parent belongs to \(\sS^{j-1}\). Otherwise, its parent would not be subdivided, i.e., it would be a stopping cube and then the children must not exist.

{We now define optimal planes for $T$ in cubes of the types defined above as follows.}

\begin{defi*}[Optimal planes]
{We say that the $m$-plane $\pi$ is an \emph{optimal plane for $T$ in $\ball{p}{r}$} if the following conditions hold true:}
\begin{itemize}
    \item {$\pi$ optimizes the excess, i.e., $ \bE(T,\ball{p}{r}) := \min_{\varkappa}\bE(T,\ball{p}{r},\varkappa) = \bE(T,\ball{p}{r},\pi),$}

    \item {$\pi$ optimizes the height among planes that optimize the excess, i.e.,} $${\bh(T,\ball{p}{r},\pi) = \min\{\bh(T,\ball{p}{r},\varkappa): \varkappa\mbox{ optimizes the excess}\}.}$$
\end{itemize}
{Moreover, for $L\in\sS\cup\sW$, we denote $\hat{\pi}_L$ to be an optimal plane for $T$ in $\sball{L}$ and $\pi_L$ to be an $m$-plane contained in $T_{p_L}\Sigma$ that minimizes $|\hat{\pi}_L-\pi_L|$.}
\end{defi*}

{We use an abuse of notation on the balls that we work with, for instance \(\tball{p_L}{r}{\pi_H}\) is not precise since \(p_L\notin \pi_H\) can totally occur, however, we understand it as \(\tball{\bp_{\pi_H}(p_L)}{r}{\pi_H}\). This abuse of notation will make the text more fluid and clear, furthermore it does not affect the mathematics in it, since up to translations everything work properly. }

\subsection{Fine properties of the Whitney decomposition}

All results disclosed in this subsection are not proved in this paper, one can look in \cite[Subsections 4.2 and 4.3]{DS4} for the proofs. We warn the reader that to find out the relation among the constants depending on the parameters \(M_0,N_0\), it is enough to carefully keep track of all choices of constants in \cite[Subsections 4.2 and 4.3]{DS4}. We show that the set of cubes \(\sW\) defined by the stopping conditions and the \(\pi_0\)-contact set \(\bGamma\) (as in Definition \ref{D:stopping_condition}) is a Whitney decomposition in the sense of Definition  \ref{D:WhitneyDecomposition}.

\begin{prop}[Whitney decomposition]\label{P:Whitney_decomp}
Under Assumptions \ref{A:CM} and assuming \(M_0\sqrt{m}2^{7-N_0} \leq 1\), then \((\bGamma, \sW)\) is a Whitney decomposition of \([-4,4]^m\subset \pi_0\). If we additionally have that 
\begin{equation}\label{E:prop_Whitney_decomp:Cstar,varepsilon}
\begin{aligned}
    \varepsilon_{cm} &\leq  \left( \frac{1}{C_0} \left(\sqrt{m}-\frac{1}{2}\right)\right)^m \leq c_0,\\
    \gamma_h +\gamma_e &\leq 1,\\
    C_e = C_e(\gamma_e,M_0,N_0) &\geq 6^mM_0^{-m}2^{(N_0+6)(m+2-2\gamma_e)}, \\
    C_h = C_h(\gamma_h,M_0,N_0,C_e(\gamma_e)) &\geq C_e^{\frac{1}{2}}C_0\left(M_0 + 2^{N_0(1+\gamma_h)}\right),
\end{aligned}
\end{equation}
then we obtain that 
\begin{equation}\label{E:prop_W:emptygenerations}
    \sW^j = \emptyset \text{ for } j\leq N_0+6.
\end{equation}
\end{prop}
\begin{rem}
Equation \eqref{E:prop_W:emptygenerations} is clearly stating that all the cubes belonging to the first 6 generations are always subdivided into other cubes.
\qed\end{rem}
\begin{rem}
We attract the reader's attention to the following: \(C_h\) depends on \(\gamma_e\) through \(C_e\), it means \(\frac{C_h}{C_e^{1/2}}\) only depends on \(m,n,\gamma_h,M_0,N_0\).
\qed\end{rem}

A piece of essential information about our Whitney decomposition is the behavior we can extract when comparing quantities among different cubes. More specifically, we disclose estimates on how much the optimal planes \(\pi_H, H\in\sW\cup\sS\), deviate from each other and from the basis \(\pi_0\), and we give upper bound for the height function \(\bh\) of the current \(T\) in cylinders with center \(p_L\) and reference plane \(\pi_H\) for possibly different cubes \(H\) and \(L\). Last but not least, we state that the portion of the support of the current \(\spt(T)\) inside those cylinders is contained in satellite balls.

\begin{prop}[Comparisons among cubes and optimal planes]\label{P:tilting_opt_planes}
Under Assumption \ref{A:CM}. Assume that the conclusion of Proposition \ref{P:Whitney_decomp} holds and
\begin{equation}\label{E:prop-TiltOptPlanes:varepsilon}
\begin{aligned}
    \max\left\{\frac{64}{7},\frac{3}{36},\frac{3}{64},\frac{1}{36}\right \} &\leq M_0 \leq \frac{2^{-7+N_0}}{\sqrt{m}},\quad \text{and}\quad \gamma_h+\gamma_e\leq 1, \\
    \varepsilon_{cm} = \varepsilon_{cm}(M_0,N_0,C_h)&\leq c_0\min\left\{\left(M_02^{-N_0}\right)^m,\left(2^{6-N_0}M_0\right)^{m^2}, 2^{-mN_0}, \right.\\
    &\quad \left. 2^{\frac{m^2(6-N_0)}{2}-mN_0}M_0^{\frac{m^2}{2}+m}, {\max}^{-1}\left\{C_h^m, 2^{\frac{m^2(N_0-6)}{2}}M_0^{\frac{m^2}{2}}\right\}\right \}, 
\end{aligned}    
\end{equation}
then we get that
\begin{equation}\label{P:CubesOptPlanes:balls}
    \sball{H} \subset \sball{L} \subset \oball{5\sqrt{m}} \text{ for all } H, L\in\sW\cup\sS \text{ with } H\subset L.
\end{equation}
Assume further that 
\begin{equation*}
    C_e = C_e(M_0,N_0) \geq \left(2^{-6+N_0}6M_0^{-1}\right)^{\frac{m}{2}},\text{ and }
    C_h = C_h(M_0,N_0) \geq C_0 \max\left \{2^{N_0}, \left(2^{N_0-6}M_0^{-1}\right)^{\frac{m}{2}}\right\}.
\end{equation*}
If \(H,L\in\sW\cup\sS\), and either (a) \(H\subset L\), or (b) \(H\cap L\neq\emptyset\) and \(2^{-1}\ell(L)\leq\ell(H)\leq \ell(L)\), then
\begin{enumerate}[\upshape (i)]
    \item \(|\hat{\pi}_H-\pi_H| \leq C_e \bm_0^{\frac{1}{2}} \ell(L)^{1-\gamma_e}\);
    
    \item \(|\pi_H-\pi_L| \leq C_e \bm_0^{\frac{1}{2}} \ell(L)^{1-\gamma_e}\);
    
    \item \(|\pi_H-\pi_0| \leq C_e \bm_0^{\frac{1}{2}}\);

    \item \(\bh(T,\cyltilted{p_H}{36r_H}{\pi_0}) \leq C_h\bm_0^{\frac{1}{2m}}\ell(H)\) and \(\spt(T)\cap\cyltilted{p_H}{36r_H}{\pi_0}\subset\sball{H}\);

    \item \(\bh(T,\cyltilted{p_L}{36r_L}{\pi}) \leq C_h\bm_0^{\frac{1}{2m}}\ell(L)^{1+\gamma_h}\) and \(\spt(T)\cap\cyltilted{p_L}{36r_L}{\pi}\subset\sball{L}\) for $\pi = \pi_H,\hat{\pi}_H$.
\end{enumerate}
\end{prop}

We exhibit upper bounds for the excess and height of stopping cubes. We recall that the stopping conditions give only lower bounds to the excess and height.

\begin{cor}[Upper bounds for stopping cubes]\label{C:tilting_opt_planes}
Provided the conclusions of \cref{P:tilting_opt_planes} hold. For every \(L\in\sW\), we have that
\begin{align}
    C_e\bm_0\ell(L)^{2-2\gamma_e} < \bEs{L} \leq 2^{m+2-2\gamma_e}C_e\bm_0\ell(L)^{2-2\gamma_e},\label{E:upper_bound_stopping_cubes:excess}\\
    C_h\bm_0^{\frac{1}{2m}}\ell(L)^{1+\gamma_h} < \bHs{L} \leq 2^{1+\gamma_h}\left(C_h + M_0\sqrt{m}C_e\right)\bm_0^{\frac{1}{2m}}\ell(L)^{1+\gamma_h}.\label{E:upper_bound_stopping_cubes:height}
\end{align}
\end{cor}

The Constancy Lemma \cite[4.1.17]{Fed} is used to prove that the property \eqref{A:projectionQ} also holds for tilted cylinders. 

\begin{lemma}[Full projection property on tilted planes]\label{L:ConstancyLemmaTilted}
Under Assumptions \ref{A:CM}, assume the conclusions of \cref{P:tilting_opt_planes} hold. We additionally assume that 
\begin{equation}\label{E:L:ConstancyLemmaTilted:varepsilon}
    \varepsilon_{cm} = \varepsilon_{cm}(M_0,N_0,C_e)  \leq  \frac{C_0}{(C_e)^{m}} .
\end{equation}
Let \(H,L\in\sW\cup\sS\) such that either (a) \(H\subset L\), or (b) \(H\cap L\neq\emptyset\) and \(2^{-1}\ell(L)\leq \ell(H) \leq \ell(L)\). Then, for $\pi = \pi_H, \hat{\pi}_H$, we have that
\begin{equation*}{\bp_{\pi}}_{\#}\left(T\res\cyltilted{p_L}{32r_L}{\pi}\right) = Q\a{\tball{p_L}{32r_L}{\pi}}.\end{equation*}
\end{lemma}

\section{Existence of the external center manifold}\label{S:CM}

We start fixing a convolution kernel $\varrho\in C^{\infty}_{{c}}(\obaseball{1})$ which is radial and satisfies $\int\varrho =1$ and $\int |x|^2\varrho(x)\mathrm{d}x = 0$. For $t>0$, we define $\varrho_t$ as $\varrho_t(x) := t^{-m}\varrho(x/t)$. We will also always assume that $\vartheta \in C_c^{\infty}([-\frac{17}{16}, \frac{17}{16}],[0,1])$ {and is identically equal to $1$ in $[-1,1]^m$.}

\subsection{Lipschitz approximations on tilted cylinders and interpolating functions}

\begin{defi}[$\pi$-approximations]
Let $L\in\sS\cup\sW$ and $\pi$ an arbitrary $m$-dimensional plane. If $T\res\cyltilted{p_L}{32r_L}{\pi}$ is under the assumptions of \cite[Thm 2.4]{DS3} in the cylinder $\cyltilted{p_L}{32r_L}{\pi}$, then the $Q$-valued map $f_L : \tball{p_L}{8r_L}{\pi} \to \cA_{Q}(\pi^{\perp})$ given by \cite[Thm 2.4]{DS3} is called \emph{$\pi$-approximation of $T$ in $\cyltilted{p_L}{8r_L}{\pi}$}. 
\end{defi}

\begin{defi}[Smoothed average]
The single valued map $\hat{h}_L: \tball{p_L}{7r_L}{\pi} \to \pi^{\perp}$ given by $\hat{h}_L:= (\etaa\circ f_L)\ast \varrho_{\ell(L)}$ is called \emph{smoothed average of the $\pi$-approximation}.
\end{defi}

\begin{prop}[Existence of interpolating functions]
Assume the conclusions of Proposition \ref{P:tilting_opt_planes} holds true. Then, we have
\begin{enumerate}[\upshape (i)]
    \item For $\pi\in\{\pi_H,\hat{\pi}_H\}$, we have that $(\bp_{\pi})_{\#}T\res\cyltilted{p_L}{32r_L}{\pi} = Q\a{\tball{p_L}{32r_L}{\pi}}$ and $T$ satisfies the assumptions of \cite[Theorem 2.4]{DS3}.
\end{enumerate}
Furthermore, let $f_{HL}$ be the $\pi_H$-approximation of $T$ in $\cyltilted{p_L}{8r_L}{\pi_H}$ and $\hat{h}_{HL}:=(\etaa\circ f_{HL})\ast\varrho_{\ell(L)}$ it smoothed average. Set $\bar{h}_{HL}:=\bp_{T_{p_H}\Sigma}(\hat{h}_{HL})$ {which takes values on $\pi_H^{\perp}\cap T_{p_H}\Sigma$}, and $h_{HL}:= (\bar{h}_{HL}, \Psi_{p_H}\circ\bar{h}_{HL})$. We then have
\begin{enumerate}[\upshape (ii)]
    \item there is a smooth function $g^\ast_{HL}:\tball{p_L}{4r_L}{\pi_0}\to\pi_0^{\perp}$ such that $\bG_{g^\ast_{HL}} = \bG_{\hat{h}_{HL}}\res\cyltilted{p_L}{4r_L}{\pi_0}$, 

    \item there is a smooth function $g_{HL}:\tball{p_L}{4r_L}{\pi_0}\to\pi_0^{\perp}$ such that $\bG_{g_{HL}} = \bG_{h_{HL}}\res\cyltilted{p_L}{4r_L}{\pi_0}$.
\end{enumerate}
\end{prop}
\begin{rem}
Note that $\bar{h}_{HL}$'s definition includes a composition with a function depending on $\Sigma$, however the function $\bp_{T_{p_H}\Sigma}$ is fixed for each $H$, i.e., it is a fixed linear function. Hence, $\bar{h}_{HL}$ does not depend on the $C^{2,\alpha}$-regularity of $\Sigma$.
\end{rem}

\begin{defi}[Glued interpolation]
{For each $j$, we define $\sP^j:= \sS^j\cup (\cup_{i= N_0}^j\sW^i)$ and }set $\vartheta_L(x) := \vartheta(\frac{x-x_L}{\ell(L)})$ and define, for each $x\in (-4,4)^m$. The \emph{glued interpolating functions at step $j$} to be 
\begin{equation*}
    \hat{\varphi}_j := \frac{\sum_{L\in\sP^j}\vartheta_L g_L}{\sum_{L\in\sP^j}\vartheta_L}\text{ and }\varphi^\ast_j := \frac{\sum_{L\in\sP^j}\vartheta_L g^{\ast}_L}{\sum_{L\in\sP^j}\vartheta_L},
\end{equation*}
{where $g_L$ and $g^\ast$ are $g_{LL}$ and $g^\ast_{LL}$, respectively. We also define $\varphi_j(x) = ( \overline{\varphi}_j(x), \Psi(y, \overline{\varphi}_j(x)))$ where $\overline{\varphi_j}$ is the first $\bar{n}$ components of $\hat{\varphi}_j$.}
\end{defi}

\subsection{Construction of the external center manifold \texorpdfstring{$\cM^\ast$}{Lg}}

The main goal of this subsection is to show (Theorem \ref{T:ECM}) the existence of the external center manifold $\cM^{\ast}$ which is a $C^{3,\kappa}, \kappa>0$, submanifold of $\R^{m+n}$. The external center manifold does not necessarily lie within 
$\Sigma$. Furthermore, we can construct (Theorem \ref{T:CM}) a center manifold $\cM$ similarly to the classical approach (as in \cite{DS4}) that lies inside $\Sigma$, it will be, however, only of class $C^{2,\alpha}$.

\begin{thm}[Existence of the center manifold $\cM$]\label{T:CM}
Under \cref{A:CM} and \cref{A:CM:parameters}, there is a positive constant \(C=C(\gamma_e,\gamma_h, M_0,N_0, C_e, C_h)\) such that, for any \(j\geq N_0\), the glued interpolating functions \(\varphi_j\) satisfy
\begin{enumerate}[\upshape (A)]
    \item\label{I:CM:C0norm} \(\|\varphi_j\|_{C^{0}} \leq C\bm_0^{1/2m}\); 
    
    \item\label{I:CM:HolderNorm} \(\|D^2\varphi_j\|_{C^{0,\alpha}} \leq C\bm_0^{1/2}\); 

    \item\label{I:CM:coincidence} if \(H\in\sW^i\) and \(L\) is the cube concentric to \(H\) with \(\ell(L) = 9\ell(H)/8\), then \(\varphi_j = \varphi_k \) on \(L\) for any \(j,k\geq i+2\); 

    \item\label{I:CM:convergence} \(\varphi_j\) {converge} in \(C^2\) to a map \(\bphi\) and \(\cM:= \mathrm{graph}(\bphi|_{(-4,4)^m})\) is a \(C^{2,\alpha}\) submanifold of \(\Sigma\).
\end{enumerate}
\end{thm}

\begin{thm}[Existence of the external center manifold $\cM^{\ast}$]\label{T:ECM}
Under \cref{A:CM} and \cref{A:CM:parameters}, there is a positive constant \(C=C(\gamma_e,\gamma_h, M_0,N_0, C_e, C_h)\) such that, for any \(j\geq N_0\), the glued interpolating functions \(\varphi^\ast_j\) satisfy
\begin{enumerate}[\upshape (A*)]
    \item\label{I:ECM:C0norm} \(\|\varphi^\ast_j\|_{C^{0}} \leq C\bm_0^{1/2m}\); 
    
    \item\label{I:ECM:HolderNorm} \(\|D^3\varphi^\ast_j\|_{C^{0,\kappa}} \leq C\bm_0^{1/2}\); 

    \item\label{I:ECM:coincidence} if \(H\in\sW^i\) and \(L\) is the cube concentric to \(H\) with \(\ell(L) = 9\ell(H)/8\), then \({\varphi}^\ast_j = {\varphi}^\ast_k \) on \(L\) for any \(j,k\geq i+2\); 

    \item\label{I:ECM:convergence} \(\varphi^\ast_j\) {converge} in \(C^3\) to a map \(\bphi^\ast\) and \(\cM^\ast := \mathrm{graph}(\bphi^\ast|_{(-4,4)^m})\) is a \(C^{3,\kappa}\) submanifold of \(\R^{m+n}\).
\end{enumerate}
\end{thm}

In order to prove the existence of $\cM$ and $\cM^\ast$, we implement an argument similar to \cite{DS4}. We start showing how to derive both existences from the constructions estimates that we state below.

\begin{prop}[Construction estimates]\label{P:CE}
Under \cref{A:CM} and \cref{A:CM:parameters}, there is a positive constant \(C=C(\gamma_e,\gamma_h, M_0,N_0, C_e, C_h)\) such that, setting $\kappa = \min\{\alpha,\gamma_{la}\}/2$, the following hold:
\begin{enumerate}[\upshape (i)]
    \item\label{I:P:CE:i} $\|g_H\|_{C^0(B)} + \|g^{\ast}_H\|_{C^0(B)}\leq C\bm_0^{1/2m}$ and $\|g_H\|_{C^{2,\alpha}(B)} + \|g^{\ast}_H\|_{C^{3,\kappa}(B)}\leq C\bm_0^{1/2}$ for $B:=\tball{x_H}{4r_H}{\pi_0}$;
    
    \item\label{I:P:CE:ii} if $H\cap L\neq\emptyset$, then $\|g_H - g_L\|_{C^l(\tball{x_H}{r_H}{\pi_0})} \leq C\bm_0^{1/2}\ell(L)^{3+\min\{\alpha,\gamma_{la}\}-l}$ for any $l\in\{0,1,2\}$ and $\|g^{\ast}_H - g^{\ast}_L\|_{C^h(\tball{x_H}{r_H}{\pi_0})}\leq C\bm_0^{1/2}\ell(L)^{3+2\kappa-h}$ for every $h\in\{0,1,2,3\}$;
    
    \item\label{I:P:CE:iii} $|D^2g_H(x_H) - D^2g_L(x_L)|\leq C\bm_0^{1/2}|x_H - x_L|^{1+\min\{\alpha,\gamma_{la}\}}$ and we also have $|D^3g^{\ast}_H(x_H) - D^3g^{\ast}_L(x_L)| \leq C\bm_0^{1/2}|x_H - x_L|^{\kappa}$;
    
    \item\label{I:P:CE:iv} $\|g_H - y_H \|_{C^0}+\|g^{\ast}_H - y_H \|_{C^0}\leq C\bm_0^{1/2m}\ell(H)$ and $|\pi_H - T_{(x,g_H(x))}\bG_{g_H}| + |\pi_H - T_{(x,g^{\ast}_H(x))}\bG_{g^{\ast}_H}|  \leq C\bm_0^{1/2}\ell(H)^{1-\gamma_e}$ for all $x\in H$;
    
    \item\label{I:P:CE:v} if $L^{\prime}$ is a cube concentric to $L\in\sW^j$ with $\ell(L^{\prime})=9\ell(L)/8$, we have $\|\hat{\varphi}_i - g_L\|_{L^1(L^{\prime})} + \|\varphi^\ast_i - g^\ast_L\|_{L^1(L^{\prime})} \leq C\bm_0^{1/2} \ell(L)^{m+3+\min\{\alpha,\gamma_{la}\}}$ for any $i\geq j$.
\end{enumerate}
\end{prop}

\begin{proof}[Proof of \cref{T:CM} and \cref{T:ECM}]
{Firstly, we define $\sP^j(H) := \{L\in\sP^j: L\cap H\neq\emptyset\}\setminus\{H\}$,} \(\chi_H := \vartheta_{H}/(\sum_{L\in\sP^j}\vartheta_L)\) for each \(H\in\sP^j\), and $\varphi_j(x):=(\bar{\varphi}_j(x), \Psi(x,\bar{\varphi}_j(x)))$. Note that 
\begin{align}
    \sum_{H\in\sP^j}\chi_H =  {\chi_L + }\sum_{H\in\sP^j(L)}\chi_H = 1 \text{ on } [-4,4]^m, \label{E:CM:Chi=1}\\
    \| D^l\chi_H\|_{C^0} \leq C_0 \ell(H)^{-l}, \quad \forall l\in \N, \label{E:CM:boundChi}\\
    \| D^l\chi_H\|_{C^{0,\theta}} \leq \frac{C(\theta)}{ \ell(H)^{l+\theta}}. \label{E:CM:ChiInterp}
\end{align}
By definition of \(\hat{\varphi}_j\) and $\varphi_j^\ast$, we have
\begin{equation*}
    \| \hat{\varphi}_j \|_{C^0} + \| \varphi^\ast_j \|_{C^0} \leq \sum_{L\in\sP^j(H){\cup \{H\}}}\|g_L\|_{C^0}+\|g_L^\ast\|_{C^0} \overset{\eqref{I:P:CE:ii}\textup{ of Prop. } \ref{P:CE}}{\leq}  C \bm_0^{1/2m},
\end{equation*}
where the summation sign disappeared by reasons of the cardinality of \(\sP^j(H)\) is bounded by a dimensional constant \(C_0\). This and the definition of $\varphi_j$ and $\varphi_j^\ast$ conclude the proof of \eqref{I:CM:C0norm} and \hyperref[I:ECM:C0norm]{(A$^\ast$)}. From \eqref{E:CM:Chi=1}, whenever \(x\in H\), we achieve that
\begin{equation}\label{E:CM:psi}
    \hat{\varphi}_j(x) = g_H(x) + \sum_{L\in\sP^j(H)}(g_L-g_H)(x)\chi_L(x)\text{ and } \varphi^\ast_j(x) = g_H^\ast(x) + \sum_{L\in\sP^j(H)}(g_L^\ast-g_H^\ast)(x)\chi_L(x).
\end{equation}
Differentiating the first equation in \eqref{E:CM:psi}, we infer that
\begin{equation*}
\begin{aligned}
    \| D^2\hat{\varphi}_j\|_{C^{0,\alpha}} &\leq \| D^2 g_H\|_{C^{0,\alpha}} + C_0\sum_{l=0}^2 \sum_{L\in\sP^j(H)}\left(\|D^l(g_H - g_L)\|_{C^{0,\alpha}}\|D^{2-l}\chi_L\|_{C^{0}}\right.\\
    &\quad \left.+ \|D^l(g_H - g_L)\|_{C^{0}}\|D^{2-l}\chi_L\|_{C^{0,\alpha}}\right) \\
    \overset{\eqref{E:CM:boundChi},\eqref{E:CM:ChiInterp}}&{\leq} \|D^2g_H\|_{C^{0,\alpha}} + C(\alpha)\sum_{l=0}^2\sum_{L\in\sP^j(H)}\ell(L)^{l-2}\left(\|D^l(g_H - g_L)\|_{C^{0,\alpha}}\right.\\
    &\quad \left.+ \|D^l(g_H - g_L)\|_{C^{0}} \ell(L)^{-\alpha} \right)\\
    \overset{\eqref{I:P:CE:i}\textup{ of Prop. }\ref{P:CE}}&{\leq} C\bm_0^{1/2} + C(\alpha)\sum_{l=0}^2\sum_{L\in\sP^j(H)}\ell(L)^{l-2}\left(\|D^l(g_H - g_L)\|_{C^{0,\alpha}}\right.\\
    &\quad \left. + \|D^l(g_H - g_L)\|_{C^{0}} \ell(L)^{-\alpha} \right).
\end{aligned}
\end{equation*}
Notice that for each \(l\in\{0,1, 2\}\) and \(L\in\sP^j(H)\), it holds
\begin{equation*}
\begin{aligned}
    \ell(L)^{l-2}\|D^l(g_H - g_L)\|_{C^{0,\alpha}} \overset{\eqref{I:P:CE:ii}\textup{ of Prop. }\ref{P:CE}}&{\leq} C\bm_0^{1/2}\ell(L)^{1-\alpha+\min\{\alpha,\gamma_{la}\}},
\end{aligned}
\end{equation*}
whereas we also obtain
\begin{equation*}
\begin{aligned}
    \|D^l(g_H - g_L)\|_{C^{0}} \ell(L)^{l-2-\alpha} \overset{\eqref{I:P:CE:ii}\textup{ of Prop. }\ref{P:CE}}&{\leq} C\bm_0^{1/2}\ell(L)^{1-\alpha+\min\{\alpha,\gamma_{la}\}},
\end{aligned}
\end{equation*}
The last three inequalities lead to 
\begin{equation}\label{E:CM:psi-holder}
\begin{aligned}
    \| D^2\hat{\varphi}_j\|_{C^{0,\alpha}(H)} \leq C \bm_0^{1/2}{\left( 1 +\right.}\ell(L)^{1-\alpha+\min\{\alpha,\gamma_{la}\}}{\left.\right)},
\end{aligned}
\end{equation}
where we use that the cardinality of \(\sP^j(H)\) is bounded by a dimensional constant. 
Note that the estimate on the Holder norm in \eqref{E:CM:psi-holder} is only on the cube \(H\), we need to extend it uniformly to the whole cube \((-4,4)^m\) to prove \eqref{I:CM:HolderNorm}. To that end, we fix \(x,y\in(-4,4)^m\) and \(H,L\in\sP^j\) such that \(x\in H\) and \(y\in L\). 

If \(H\cap L\neq\emptyset\), the proof is trivial by \eqref{E:CM:psi-holder} and the triangle inequality. If \(H\cap L=\emptyset\), we prove the last inequality as follows. Without loss of generality, assume that \(\ell(L)\geq\ell(H)\), then by simple geometric considerations, we have
\begin{equation*}
    \max\{|x-x_H|,|y-x_L|\} \leq \sqrt{m}\ell(L) \leq 2\sqrt{m}|x-y|.
\end{equation*}
Using that \(\hat{\varphi}_j \equiv g_Y\) on a neighborhood of \(x_Y\) for any cube \(Y\in\sP^j\), we thus obtain
\begin{equation*}
\begin{aligned}
    | D^2\hat{\varphi}_j(x) - D^2\hat{\varphi}_j(y) | &\leq | D^2\hat{\varphi}_j(x) - D^2\hat{\varphi}_j(x_H) | + | D^2g_H(x_H) - D^2g_L(x_L) |  + | D^2\hat{\varphi}_j(x_L) - D^2\hat{\varphi}_j(y) |\\
    \overset{\eqref{I:P:CE:iii}\textup{ of Prop. }\ref{P:CE},\eqref{E:CM:psi-holder}}&{\leq} C\bm_0^{1/2}|x_H-x_L|^{1-\alpha+\min\{\alpha,\gamma_{la}\}}.
\end{aligned}    
\end{equation*}
This inequality and an easy application of the chain rule on $\varphi_j$ finish the proof of \eqref{I:CM:HolderNorm}. Notice that the proof of \hyperref[I:ECM:HolderNorm]{(B$^\ast$)} goes in the very same lines of the four previous displayed equations with the use of the bound on the \emph{third} derivatives of $g_H^\ast$ and $g_L^\ast$ provided by \cref{P:CE} in items \eqref{I:P:CE:i}, \eqref{I:P:CE:ii}, and \eqref{I:P:CE:iii}. 

Fix \(H\in\sW^i\) and \(j\geq i+2\), by construction, we have \(\sP^j(H)=\sP^{i+2}(H)\subset\sW\). If \(L\) is the cube concentric to \(H\) with \(\ell(L) = 9\ell(H)/8\), recalling the definition of \(\vartheta_Y\), we easily ensure \(\spt(\vartheta_Y)\cap L=\emptyset\) for all cubes \(Y\notin\sP^j(H)\). Hence, \eqref{I:CM:coincidence} and \hyperref[I:CM:coincidence]{(C$^\ast$)} are verified. We finally start the proof of \eqref{I:CM:convergence}. Differentiating the first equality in \eqref{E:CM:psi} and putting \eqref{E:CM:boundChi} into account, for \(l\in\{0,1,2\}\), we obtain
\begin{equation}\label{E:CM:psi-derivada}
\begin{aligned}
    \|D^l\hat{\varphi}_j\|_{C^{0}} &\leq \|D^lg_H\|_{C^{0}} + \sum_{k=0}^l\sum_{L\in\sP^j(H)}\|D^k(g_H-g_L)\|_{C^{0}}\ell(L)^{k-l}   \\
    \overset{\eqref{I:P:CE:i}\textup{ of Prop. } \ref{P:CE}}&{\leq} C\bm_0^{1/2} + \sum_{k=0}^l\sum_{L\in\sP^j(H)}\|D^k(g_H-g_L)\|_{C^{0}}\ell(L)^{k-l}\\
    \overset{\eqref{I:P:CE:iii}\textup{ of Prop. } \ref{P:CE}}&{\leq} 
     C\bm_0^{1/2}.
\end{aligned}
\end{equation}
Pick a point \(x\in (-4,4)^m\) such that \(x\in H\cap L\) with \(L\in\sP^j\) and \(H\in\sP^{j+1}\). Notice that
\begin{equation}\label{E:CEII:quase-la}
    \hat{\varphi}_j(x_H) = g_H(x_H)\text{ and } \hat{\varphi}_{j+1}(x_L)=g_L(x_L),
\end{equation}
thus the following holds
\begin{equation*}
\begin{aligned}
    |\hat{\varphi}_j(x) - \hat{\varphi}_{j+1}(x)| &\leq |\hat{\varphi}_j(x) - \hat{\varphi}_j(x_H)| + |g_H(x_H) - g_L(x_L)| + |\hat{\varphi}_{j+1}(x_L) - \hat{\varphi}_{j+1}(x)| \\
    &\leq C_0 \left(\|D\hat{\varphi}_j\|_{C^{0}} + \|D\hat{\varphi}_{j+1}\|_{C^0(L)} \right)\ell(L) + |g_H(x_H) - y_H| \\
    &\quad + |g(x_L) - y_L| + |p_H - p_L| \\
    \overset{
    \eqref{E:CM:psi-derivada},(\ast)}&{\leq}  C\bm_0^{1/2}\ell(L) = C2^{-j},
\end{aligned}
\end{equation*}
in \((\ast)\) we use the same computation as in 
\cite[Eq. 4.46]{Spolaor}. Notice that the same bound holds to $\varphi_j$ due to the Lipschitz continuity of $\Psi$. So, passing the last inequality to the limit, we obtain that \((\varphi_j)_{j\geq N_0}\) uniformly converges to a map \(\bphi\). It is straightforward to derive the \(C^2\) convergence, indeed it is concluded as follows
\begin{equation*}
\begin{aligned}
    |D^2\hat{\varphi}_j(x) - D^2\hat{\varphi}_{j+1}(x)| &\overset{\eqref{E:CEII:quase-la}}{\leq} \frac{|D^{{2}}\hat{\varphi}_j(x) - D^2\hat{\varphi}_j(x_H)|}{|x-x_H|^{\alpha}}|x-x_H|^{\alpha} + |D^2g_H(x_H) - D^2g_L(x_L)| \\
    &\quad + \frac{|D^2\hat{\varphi}_{j+1}(x_L) - D^2\hat{\varphi}_{j+1}(x)|}{|x-x_H|^{\alpha}}|x-x_H|^{\alpha} \\
    \overset{\eqref{I:P:CE:iii}\textup{ of Prop. }\ref{P:CE},\eqref{E:CM:psi-holder}}&{\leq} C\ell(L)^{1+\min\{\alpha,\gamma_{la}\}} = C2^{-j(1+\min\{\alpha,\gamma_{la}\})}.
\end{aligned}
\end{equation*}
The chain rule and the definition of $\varphi_j$ assures that this bound is also true for $\varphi_j$. Again, using the same argument and the bounds on the third derivatives given in \cref{P:CE}, we obtain 
\begin{equation*}
    |D^3\varphi^\ast_j(x) - D^3\varphi^\ast_{j+1}(x)| \leq C2^{-2j\kappa}.
\end{equation*}
At last, the Holder regularity of the external center manifold \(\cM^\ast:=\mathrm{graph}(\bphi^\ast|_{(-4,4)^m})\) and of the center manifold \(\cM:=\mathrm{graph}(\bphi|_{(-4,4)^m})\) are a consequence of the convergences above and \hyperref[I:CM:HolderNorm]{(B$^\ast$)} and \eqref{I:CM:HolderNorm}, respectively.
\end{proof}

\subsection{Proof of the construction estimates for \texorpdfstring{$\cM^\ast$}{Lg} through the elliptic system}

\begin{defi}[Tangential parts]
Fix $H\in\sP^j$ and let $\varkappa$ be the orthogonal complement of $\pi_H$ in $T_{p_H}\Sigma$. Given $p\in\R^{m+n}$, any set $\Omega\subset\pi_H$, and any function $\chi:p+\Omega\to\pi_H^{\perp}$, the map $\bp_\varkappa\circ \chi$ will be called the \emph{tangential part of $\chi$} and denoted by $\bar{\chi}$.
\end{defi}

As it is well-known, from the minimality of $T$, one can derive an elliptic systems that leads to great estimates on the $\pi_L$-approximations that are gathered in \cref{P:PDE}. These estimates in turn lead (not trivially, see Lemmas \ref{L:CjEstimates} and \ref{L:L1estimates}) to the bounds on the derivatives given in the construction estimates (\cref{P:CE}).

We mention that the following proof differs from the proof presented in \cite{DS4}, since the authors rely on the $C^{3,\varepsilon_0}$-regularity of $\Sigma$ to get the bounds on the error terms that will appear naturally when investigating the first variation of the current $\bG_{f_{HL}}$ (see \eqref{E:P:PDE:several-quantities}). 

\begin{prop}[Elliptic system]\label{P:PDE}
Assume \cref{A:CM} and \cref{A:CM:parameters}, we denote $B:=\tball{p_L}{8r_L}{\pi_H}$ and $\sigma := \min\{\alpha,\gamma_{la}\}$. There is a positive constant \(C=C(\gamma_e,\gamma_h, M_0,N_0, C_e, C_h)\) such that
\begin{equation}\label{E:P:PDE}
    \left|\int_{B} D(\etaa\circ\bar{f}_{HL}):D\zeta + (\bp_{\pi_H}(x-p_H)^t)\cdot \bL\cdot \zeta \right|\leq C\bm_0r_H^{{m+}1+\sigma}\left( r_H\|\zeta\|_{C^1(B)} + \|\zeta\|_{C^0(B)}\right),
\end{equation}
for any test function $\zeta$. Moreover 
\begin{equation}\label{E:P:PDE-moreover}
    \| \bar{h}_{HL} - \etaa\circ\bar{f}_{HL} \|_{L^1(\tball{x_L}{7r_L}{\pi_H})} \leq C\bm_0 r_L^{m+3+\sigma}.
\end{equation}
\end{prop}
\begin{proof}
We will denote geometric constants by $C_0$, whereas $C$ denotes constants depending upon the parameters $\gamma_h, \gamma_e, M_0, N_0, C_e$ and $C_h$. In order to simplify the notation, we fix a system of coordinates $(x, y, z) \in \pi_H \times \varkappa \times\left(T_{p_H} \Sigma\right)^{\perp}$ so that $p_H=(0,0,0)$. Although the domains of the various maps are subsets $\Omega$ of $p_L + \pi_H$, from now on we will consider them as functions of $x$; i.e., we shift their domains to $\bp_{\pi_H}(\Omega)$. We also use $\Psi_H$ for the map $\Psi_{p_H}$ which graph gives \(\Sigma\) as in \cref{A:CM}. Recall that $\Psi_H(0,0)=0$, $D \Psi_H(0,0)=0$ and $\left\| \Psi_H\right\|_{C^{2, \alpha}} \leq \bm_0^{1 / 2}$.

Given a test function $\zeta$ and any point $q=(x, y, z) \in \Sigma$, we consider the vector field $\chi(q)=\left(0, \zeta(x), D_y \Psi(x, y) \cdot \zeta(x)\right)$. The vector field $\chi$ is tangent to $\Sigma$, and therefore $\delta T(\chi)=0$. Thus, we easly conclude that
\begin{equation}\label{E:P:PDE:1var}
    \left|\delta \bG_{f_{HL}}(\chi)\right|=\left|\delta \bG_{f_{HL}}(\chi)-\delta T(\chi)\right| \leq C_0 \int_{\cyltilted{p_L}{8r_L}{\pi_H}}|D \chi| d\left\|\bG_{f_{HL}}-T\right\|. 
\end{equation}
Since we have \(\|D\Psi_H\|\leq \bm_0^{1/2}\), choosing $\varepsilon_2$ sufficiently small, we achieve \begin{equation}
    |\chi|\leq 2|\zeta| \text{ and } |D\chi|\leq 2|D\zeta|+2|\zeta|.
\end{equation}
We denote $\bE_{HL}:=\bE(T,\cyltilted{p_L}{32r_L}{\pi_H}$ and recall the estimates in \cite[Theorem 1.4]{DS3}, where $K_{HL} \subset \tball{p_L}{8r_L}{\pi_H}$ is the {good} set, {namely}
\begin{align}
    |D f_{HL}| &\leq C_0 \bE_{HL}^{\gamma_{la}}+C_0 r \bA \leq C \bm_0^{\gamma_{la}} r_H^{\gamma_{la}}, \label{E:P:PDE:boundDf}\\
    |f_{HL}| &\leq C_0 \bh\left(T, \cyltilted{p_L}{32r_L}{\pi_H}\right)+C_0\left(\bE_{HL}^{1 / 2}+r_H \bA\right) r_H \leq C \bm_0^{\frac{1}{2m}} r_H^{1+\gamma_h}, \label{E:P:PDE:C0bound_lipapprox}\\
    \int_{\tball{p_L}{8r_L}{\pi_H}}|D f_{HL}|^2 &\leq C_0  \bE_{HL}r_H^m \leq C \bm_0 r_H^{m+2-2 \gamma_e} \label{E:P:PDE:boundDir}, \\
    |\tball{p_L}{8r_L}{\pi_H} \setminus K_{HL}| &\leq C_0 \bE_{HL}^{\gamma_{la}}\left(\bE_{HL}+r_H^2 \bA^2\right)r_H^m \leq C \bm_0^{1+\gamma_{la}} r_H^{m+(2-2 \gamma_e)(1+\gamma_{la})}, \\
    \biggl|\biggr.\|T\|\left(\cyltilted{p_L}{8r_L}{\pi_H}\right) &-|\tball{p_L}{8r_L}{\pi_H}| - \frac{1}{2} \int_{\tball{p_L}{8r_L}{\pi_H}}| D f_{HL}|^2 \biggl.\biggr| \\
    &\quad \leq C_0 \bE_{HL}^{\gamma_{la}}\left(\bE_{HL}+r_H^2 \bA^2\right)r_H^m \leq C \bm_0^{1+\gamma_{la}} r_H^{m+(2-2 \gamma_e)(1+\gamma_{la})}. \nonumber
\end{align}
Concerning \eqref{E:P:PDE:C0bound_lipapprox} observe that the statement of \cite[Theorem 1.4]{DS3} indeed bounds $\osc(f_{HL})$. Moreover, in our case we have $p_H=(0,0,0) \in \spt(T)$ and $\spt(T) \cap$ $\bG_{f_{HL}} \neq \emptyset$. Thus we conclude $|f_{HL}| \leq C_0 \osc(f_{HL})+C_0 \bh\left(T, \cyltilted{p_L}{32r_L}{\pi_H}\right)$. Taking selections $f_{HL}=\sum_{i}\a{f_{i,HL}}$ and $\bar{f}_{HL}=\sum_{i}\a{\bar{f}_{i,HL}}$, we have that $\spt(\bG_{f_{HL}})\subset\Sigma$ implies
\begin{equation}
    f_{HL} = \sum_{i=1}^{Q}\a{(\bar{f}_{i,HL}, \Psi_H(\bar{f}_{i,HL}) )}.
\end{equation} 
We have from \cite[Theorem 4.1]{DS2} (straightforwardly estimating $\mathrm{Err}$ with \eqref{E:P:PDE:boundDf} and \eqref{E:P:PDE:boundDir}) the following inequality for the first variation
\begin{equation}\label{E:P:PDE:several-quantities}
\begin{aligned}
    \delta\bG_{f_{HL}}(\chi) &= \mathrm{Err} + \int_{\tball{p_L}{8r_L}{\pi_H}}\sum_{i=1}^{Q}\left(S_1 + S_2 + S_3\right):\left(S_4+S_5\right) + D_x\zeta:D_x\bar{f}_{i,HL},\\
    \mathrm{Err} &\leq C\|\zeta\|_{C^1}\bm_0^{1+\gamma_{la}}r_H^{m+2-2\gamma_e+\gamma_{la}},
\end{aligned}
\end{equation}
where we are using the following notations
\begin{equation*}
\begin{aligned}
    S_1 &:= D_{xy}\Psi_H(x,\bar{f}_{i,HL}(x))\cdot\zeta, \quad S_2 := (D_{yy}\Psi_H(x,\bar{f}_{i,HL}(x))\cdot D_x\bar{f}_{i,HL})\cdot \zeta \\
    S_3 &:= D_{y}\Psi_H(x,\bar{f}_{i,HL}(x))\cdot D_x\zeta, \quad S_4 := D_x\Psi_H(x,\bar{f}_{i,HL}(x)), \quad 
    S_5 := D_y\Psi_H(x,\bar{f}_{i,HL}(x))\cdot D_x\bar{f}_{i,HL}. 
\end{aligned}
\end{equation*}
Our goal now is to use \eqref{E:P:PDE:several-quantities} and \eqref{E:P:PDE:1var} to prove the proposition. To that end, we will estimate all related terms. Let us start with a Taylor expansion for $\Psi_H\in C^{2,\alpha}$ using \cite[Proposition 2.1]{berger2020quality} and the aforementioned considerations for $\Psi_H$ to derive the following bound
\begin{align*}
    |D\Psi_H(x,y) - D_xD\Psi_H(0,0)\cdot x -D_yD\Psi_H(0,0)\cdot y | &\leq C_0 \| \Psi_H\|_{C^{2,\alpha}}(|x|^2 + |y|^2)^{1/2+\alpha/2},\\
    |D^2\Psi_H(x,y) - D^2\Psi_H(0,0) | &\leq C_0 \| \Psi_H\|_{C^{2,\alpha}}(|x| + |y|)^{\alpha}.
\end{align*}
By \eqref{E:P:PDE:C0bound_lipapprox} we have $\|(x,\bar{f}_{i,HL}(x))\|\leq C\|(x,f_{i,HL})\|\leq C|x|$. Joining the last fact with $\| \Psi_H\|_{C^{2,\alpha}}\leq \bm_0^{1/2}$ and the last displayed inequalities, we infer the following 
\begin{align}
    |D\Psi_H(x,\bar{f}_{i,HL}(x)) - D_xD\Psi_H(0,0)\cdot x | &= O\left(\bm_0^{1/2}r_H^{1+\alpha}\right) + O\left(\bm_0^{1/2}r_H^{1+\alpha}\right), \label{E:P:PDE:DpsiO}\\
    |D\Psi_H(x,\bar{f}_{i,HL}(x))| &\leq \bm_0^{1/2}r_H, \label{E:P:PDE:Dpsi}\\
    |D^2\Psi_H(x,\bar{f}_{i,HL}(x)) - D^2\Psi_H(0,0) | &= O\left(\bm_0^{1/2}r_H^{\alpha}\right) \label{E:P:PDE:D2psiO},\\
    |D^2\Psi_H(x,\bar{f}_{i,HL}(x))|&\leq  \bm_0^{1/2}\label{E:P:PDE:D2psi}.
\end{align}
We now start to estimate the several quantities appearing in \eqref{E:P:PDE:several-quantities}. Henceforth we omit the domain of integration $\tball{p_L}{8r_L}{\pi_H}$ for the sake of simplicity. We begin as follows
\begin{equation}\label{E:P:PDE:S1eq1}
\begin{aligned}
    \int\sum_{i=1}^Q S_1:S_4 \overset{\eqref{E:P:PDE:D2psi},\eqref{E:P:PDE:DpsiO}}&{=} \int\sum_{i=1}^Q  (D_{xy}\Psi_H(x,\bar{f}_{i,HL}(x))\cdot\zeta):( D_{x}D\Psi_H(0,0)\cdot x)\\
    &\quad + O\left(\bm_0r_H^{1+\min\{\alpha,\gamma_{la}\}}\int|\zeta|\right) \\
    \overset{\eqref{E:P:PDE:D2psiO},\eqref{E:P:PDE:Dpsi}}&{=}   \int\sum_{i=1}^Q  (D_{xy}\Psi_H(0,0)\cdot\zeta):( D_{x}D\Psi_H(0,0)\cdot x)\\
    &\quad + O\left(\bm_0r_H^{1+\min\{\alpha,\gamma_{la}\}}\int|\zeta|\right) \\
    &=: O\left(\bm_0r_H^{1+\min\{\alpha,\gamma_{la}\}}\int|\zeta|\right) 
    + \int  x^t\cdot\bL_{1,4}\cdot\zeta,
\end{aligned}
\end{equation}
where $\bL_{1,4}=\bL_{1,4}(D^2\Psi_H(0,0))$ is defined in the last inequality and we bring the reader's attention to the fact that it does not depend on $L$. It is easy to see that $\bL_{1,4}$ is a quadratic form of $D^2\Psi_H(0,0)$. We also have
\begin{equation}\label{E:P:PDE:S1eq2}
\begin{aligned}
    \int\sum_{i=1}^Q S_1:S_5 \overset{\eqref{E:P:PDE:D2psi},\eqref{E:P:PDE:Dpsi},\eqref{E:P:PDE:boundDf}}&{=}  O\left(\bm_0^{1+\gamma_{la}}r_H^{1+\gamma_{la}}\int|\zeta|\right).
\end{aligned}
\end{equation}
Now that we handled the terms involving $S_1$, let us turn the attention to $S_2$. We have
\begin{equation}\label{E:P:PDE:S2}
    \int\sum_{i=1}^Q S_2:(S_4+S_5) \overset{\eqref{E:P:PDE:D2psi},\eqref{E:P:PDE:Dpsi},\eqref{E:P:PDE:boundDf}}{=} O\left(\bm_0^{1+\gamma_{la}}r_H^{1+\gamma_{la}}\int|\zeta|\right).
\end{equation}
It remains to take care of $S_3$, we proceed as follows
\begin{equation}\label{E:P:PDE:S3eq1}
    \int\sum_{i=1}^Q S_3:S_5 \overset{\eqref{E:P:PDE:D2psi},\eqref{E:P:PDE:Dpsi},\eqref{E:P:PDE:boundDf}}{=} O\left(\bm_0^{1+\gamma_{la}}r_H^{2+\gamma_{la}}\int|D\zeta|\right).
\end{equation}
Additionally, we obtain
\begin{equation}\label{E:P:PDE:S3eq2}
\begin{aligned}
    \int\sum_{i=1}^Q S_3:S_4 \overset{\eqref{E:P:PDE:DpsiO},\eqref{E:P:PDE:Dpsi}}&{=} \int\sum_{i=1}^Q ((D_{yx}\Psi_H(0,0)\cdot x )\cdot D_x\zeta) : D_x\Psi_H(x,\bar{f}_{i,HL}(x))\\
    &\quad + O\left( \bm_0 r_H^{2+\min\{\alpha,\gamma_{la}\}}\int|D\zeta|\right)\\
    \overset{\eqref{E:P:PDE:DpsiO},\eqref{E:P:PDE:D2psi}}&{=} \int\sum_{i=1}^Q ((D_{yx}\Psi_H(0,0)\cdot x )\cdot D_x\zeta) : (D_{xx}\Psi_H(0,0)\cdot x) \\
    &\quad +O\left( \bm_0 r_H^{2+\min\{\alpha,\gamma_{la}\}}\int|D\zeta|\right)\\
    \overset{(\ast)}&{=} O\left( \bm_0 r_H^{2+\min\{\alpha,\gamma_{la}\}}\int|D\zeta|\right) + \int x^t\cdot \bL_{3,4}\cdot \zeta,
\end{aligned}
\end{equation}
where in $(\ast)$ we integrate by parts and define $\bL_{3,4} = \bL_{3,4}(D^2\Psi_H(0,0))$ in accordance with the equality. As before, we bring the reader's attention to the fact that $\bL_{3,4}$ does not depend on $L$ and is a quadratic form of $D^2\Psi_H(0,0)$. Denoting $\bL=\bL(D^2\Psi_H(0,0)):= \bL_{1,4}+\bL_{3,4}$ and putting together \eqref{E:P:PDE:several-quantities}, \eqref{E:P:PDE:S1eq1}, \eqref{E:P:PDE:S1eq2}, \eqref{E:P:PDE:S2}, \eqref{E:P:PDE:S3eq1}, and \eqref{E:P:PDE:S3eq2}, we obtain that
\begin{equation*}
\begin{aligned}
    \delta\bG_{f_{HL}}(\chi) &= \int\left(D_x\etaa\circ\bar{f}_{HL}:D_x\zeta + x^t\cdot \bL\cdot \zeta \right) + O\left(\|\zeta\|_{C^1}\bm_0^{1+\gamma_{la}}r_H^{m+2-2\gamma_e+\gamma_{la}}\right) \\
    &\quad + O\left(\bm_0r_H^{{m+}1+\min\{\alpha,\gamma_{la}\}}\|\zeta\|_{L^1}\right) + O\left(\bm_0r_H^{{m+}2+\alpha}\|D\zeta\|_{L^1}\right).
\end{aligned}
\end{equation*}
Such equation inserted into \eqref{E:P:PDE:1var} generates
\begin{equation}\label{E:P:PDE:0.5}
\begin{aligned}
    \left|\int D_x\etaa\circ\bar{f}_{HL}:D_x\zeta + x^t\cdot \bL\cdot \zeta \right|&\leq O\left(\|\zeta\|_{C^1}\bm_0^{1+\gamma_{la}}r_H^{m+2-2\gamma_e+\gamma_{la}}\right) + O\left(\bm_0r_H^{{m+}2+\alpha}\|D\zeta\|_{L^1}\right) \\
    &\quad + O\left(\bm_0r_H^{{m+}1+\min\{\alpha,\gamma_{la}\}}\|\zeta\|_{L^1}\right) + C_0 \int|D \zeta| d\vartheta,    
\end{aligned}
\end{equation}
where we are using for any Borel set $E$
\begin{equation*}
    \vartheta(E):=\cH^m(E\setminus K_{HL}) + \|T\|((E\setminus K_{HL})\times\R^n) \text{ and } \|T-\bG_{f_{HL}}\|(E\times\R^n) \leq C_0\vartheta(E).
\end{equation*}
The same argument provided in \cite[Eq. 5.20]{DS4} ensures that
\begin{equation*}
    \vartheta(\tball{p_L}{8r_L}{\pi_H}) \leq C\bm_0r_H^{m+2-2\gamma_e+\gamma_{la}}.
\end{equation*}
By \eqref{E:P:PDE:0.5} and the last inequality, we derive that
\begin{equation}
\begin{aligned}
    \left|\int D_x\etaa\circ\bar{f}_{HL}:D_x\zeta + x^t\cdot \bL\cdot \zeta \right|&\leq O\left(\|\zeta\|_{C^1}\bm_0 r_H^{m+2-2\gamma_e+\gamma_{la}}\right) + O\left(\bm_0r_H^{{m+}2+\alpha}\|D\zeta\|_{L^1}\right) \\
    &\quad + O\left(\bm_0r_H^{{m+}1+\min\{\alpha,\gamma_{la}\}}\|\zeta\|_{L^1}\right).    
\end{aligned}
\end{equation}
Thanks to the choice of the exponents, we have $\gamma_{la} - 2\gamma_e = 49\gamma_{la}/50 < 0$, hence \eqref{E:P:PDE} follows. The proof of the moreover part goes along the very same line of the proof of \cite[Eq. 5.2]{DS4} using \eqref{E:P:PDE}, we omit it here.
\end{proof}

For the sake of clarity and completeness, we will state the following two lemmas that will be used in the proof of \ref{P:CE}.

\begin{lemma}[From the elliptic system to $C^j$ estimates]\label{L:CjEstimates}
Assume \cref{A:CM} and \cref{A:CM:parameters} and set $B^{\prime}:= \tball{p_H}{5r_H}{\pi_H}$ and $B:= \tball{p_H}{4r_H}{\pi_H}$. Then, we obtain that
\begin{align}
    \|\bar{h}_{HL}-\bar{h}_H\|_{C^j(B^{\prime})}  + \|g^{\ast}_{HL}-g^{\ast}_H\|_{C^j(B)} &\leq C\bm_0^{1/2}\ell(L)^{3+\min\{\alpha,\gamma_{la}\}-j},\quad j\in\{0,1,2,3\}, \label{E:Ck:h-bar1}\\
    \|\bar{h}_{HL}-\bar{h}_H\|_{C^{3,\theta}(B^{\prime})}  + \|g^{\ast}_{HL}-g^{\ast}_H\|_{C^{3,\theta}(B^{\prime})} &\leq C\bm_0^{1/2}\ell(L)^{\min\{\alpha,\gamma_{la}\}-\theta}, \quad \forall \theta\in (0,1),
    \label{E:Ck:h-bar2}\\
    \|h_{HL}-h_H\|_{C^j(B^{\prime})} + \|g_{HL}-g_H\|_{C^j(B)} &\leq C\bm_0^{1/2}\ell(L)^{3+\min\{\alpha,\gamma_{la}\}-j},\quad j\in\{0,1,2\}, \label{E:Ck}\\
    \|h_{HL}-h_H\|_{C^{2,\alpha}(B^{\prime})} + \|g_{HL}-g_H\|_{C^{2,\alpha}(B)} &\leq C\bm_0^{1/2}\ell(L)^{1-\alpha+\min\{\alpha,\gamma_{la}\}} \label{E:Ck,alpha}.
\end{align}
Consequently, items \eqref{I:P:CE:i} and \eqref{I:P:CE:iv} of \cref{P:CE} hold.
\end{lemma}
\begin{rem}
Notice that $\bar{h}_{HL}$ and $g^{\ast}_{HL}$ are proven to be in fact of class $C^{3,\min\{\alpha,\gamma_{la}\}}$. On the other hand, since $h_{HL}$ is a composition of $\Psi_H$ and $\bar{h}_{HL}$, the regularity of $h_{HL}$ (and consequently of $g_{HL}$) do not always exceed $C^{2,\alpha}$ which is $\Sigma$'s regularity. 
\end{rem}
\begin{proof}
The proof of \cite[Lemma 5.3]{DS4} works in the same lines using \cref{P:PDE} instead of \cite[Proposition 5.2]{DS4}.
\end{proof}

\begin{lemma}[Tilted $L^1$ estimates]\label{L:L1estimates}
Assume \cref{A:CM} and \cref{A:CM:parameters}. We have that
\begin{equation}
    \|h_{HJ} - \hat{h}_{LM}\|_{L^1(\tball{p_J}{2r_J}{\pi_H})} \leq C\bm_0\ell(J)^{m+3+\min\{\alpha,\gamma_{la}\}}.
\end{equation}
\end{lemma}
\begin{proof}
The very same proof given in \cite[Lemma 5.5]{DS4}.
\end{proof}

We are now able to derive the construction estimates for the functions $g_{HL}$ and $g_{HL}^\ast$.

\begin{proof}[Proof of \cref{P:CE}]
Since $\kappa = \min\{\alpha,\gamma_{la}\}/2$, recall that \eqref{I:P:CE:i} and \eqref{I:P:CE:iv} are proven in \cref{L:CjEstimates}. Let us show how to prove \eqref{I:P:CE:ii}. Take  $H,L\in\sP^j$ with nonempty intersection. We show that the inequality
\begin{equation}\label{E:CE:proof1}
    \|h_{H} - \hat{h}_{L}\|_{L^1(\tball{p_H}{2r_H}{\pi_H})} \leq C\bm_0^{1/2}\ell(L)^{m+3+\min\{\alpha,\gamma_{la}\}}
\end{equation}
holds true. If $\ell(L) = \ell(H)$, it is a direct consequence of \cref{L:L1estimates}. If $\ell(L) = 2\ell(H)$, take $J$ to be the father of $H$. It is clear that $J\cap L\neq\emptyset$, therefore we can apply \cref{L:L1estimates} to obtain
\begin{equation*}
      \|h_{HJ} - \hat{h}_{L}\|_{L^1(\tball{p_H}{2r_H}{\pi_H})} \leq C\bm_0\ell(L)^{m+3+\min\{\alpha,\gamma_{la}\}}.
\end{equation*}
On the other hand, by \eqref{E:Ck}, we have
\begin{equation*}
    \|h_{H} - h_{HJ}\|_{L^1(\tball{p_H}{2r_H}{\pi_H})} \leq C\bm_0^{1/2}\ell(L)^{m+3+\min\{\alpha,\gamma_{la}\}}.
\end{equation*}
The last two inequalities prove the validity of \eqref{E:CE:proof1} in the latter case. Recalling that by construction, we have $\bG_{g_X}\res\cyltilted{x_H}{r_H}{\pi_0} = \bG_{h_{X}}\res\cyltilted{x_H}{r_H}{\pi_0} $ (same with $g^{\ast}$ and $\hat{h}$) for $X\in\{L,H\}$. Thus, as a consequence of \cite[Lemma B.1]{DS4} we deduce
\begin{equation}\label{E:CE:proof2}
    \| g_H - g_L\|_{L^1(\tball{x_H}{r_H}{\pi_0})} + \| g^{\ast}_H - g^{\ast}_L\|_{L^1(\tball{x_H}{r_H}{\pi_0})}\leq C\bm_0^{1/2} \ell(L)^{m+3+\min\{\alpha,\gamma_{la}\}}.
\end{equation}
The fact that $\|g_H-g_L\|_{C^{2,\alpha}(\tball{x_H}{r_H}{\pi_0})}\leq C\bm_0^{1/2}\ell(L)^{1-\alpha+\min\{\alpha,\gamma_{la}\}}$ (see \eqref{E:Ck,alpha}) and the last inequality together with
\begin{equation*}
    \| D^j(g_H - g_L) \|_{C^0} \leq Cr_L^{-m-j}\|g_H-g_L\|_{L^1} + Cr_L^{2+\alpha-j}\|D^2(g_H-g_L)\|_{C^{0,\alpha}}
\end{equation*}
imply \eqref{I:P:CE:ii}, for $g^{\ast}_H$ and $g^{\ast}_L$ the proof is exactly the same using \eqref{E:Ck:h-bar2} and
\begin{equation*}
    \| D^j(g^{\ast}_H - g^{\ast}_L) \|_{C^0} \leq Cr_L^{-m-j}\|g^{\ast}_H-g^{\ast}_L\|_{L^1} + Cr_L^{3+\kappa-j}\|D^3(g^{\ast}_H-g^{\ast}_L)\|_{C^{0,\kappa}}.
\end{equation*}
We now prove \eqref{I:P:CE:v} since it is an easy consequence of \eqref{E:CE:proof1} instead of \eqref{E:Ck,alpha}. Indeed, if $L\in\sW^j$ and $i\geq j$, consider the subset $\sP^i(L):=\{X\in\sP^i: X\cap L\neq\emptyset\}$. Note that the cardinality of $\sP^i(L)$ is bounded by a dimensional constant. If $L^{\prime}$ is a cube concentric to $L$ with $\ell(L^{\prime}) = 9\ell(L)/8$, by definition of $\hat{\varphi}_j$ and the last considerations, we have
\begin{equation*}
\begin{aligned}
    \|\hat{\varphi}_i - g_L\|_{L^1(L^{\prime})} + \|\varphi^\ast_i - g^{\ast}_L\|_{L^1(L^{\prime})} &\leq C\sum_{X\in\sP^i(L)}\left(\| g_H - g_L\|_{L^1(\tball{x_X}{r_X}{\pi_0})}+\| g^{\ast}_H - g^{\ast}_L\|_{L^1(\tball{x_X}{r_X}{\pi_0})}\right) \\
    \overset{\eqref{E:CE:proof2}}&{\leq} C\bm_0^{1/2} \ell(L)^{m+3+\min\{\alpha,\gamma_{la}\}},    
\end{aligned}
\end{equation*}
which is exactly \eqref{I:P:CE:v}. Aiming at proving \eqref{I:P:CE:iii}, we take $J$ to be any ancestor of $H$ and $H_0=H,\ldots,H_{i_0}=J$ a chain of cubes such that $H_{i}$ is the father of $H_{i-1}$. Then, we have
\begin{equation}\label{E:CE:proof3}
    |D^2g_H(x_H) - D^2g_J(x_J) |\leq \sum_{i=0}^{i_0} |D^2g_{H_i}(x_{H_i}) - D^2g_{H_{i+1}}(x_{H_{i+1}}) | \overset{\eqref{E:Ck,alpha}}{\leq} C\bm_0^{1/2}\ell(J)^{1+\min\{\alpha,\gamma_{la}\}}.
\end{equation}
If we now take any random par of cubes $H,L\in\sP^j$, we choose $J_H$ and $J_L$ to be the first ancestors of $H$ and $L$, respectively, such that $J_H\cap J_L\neq\emptyset$. We then have
\begin{equation*}
\begin{aligned}
    |D^2g_H(x_H) - D^2g_L(x_L) | &\leq  |D^2g_{H}(x_{H}) - D^2g_{J_H}(x_{J_H}) | + |D^2g_{J_H}(x_{J_H}) - D^2g_{J_L}(x_{J_L}) | \\
    &\quad + |D^2g_{L}(x_{L}) - D^2g_{J_L}(x_{J_L}) | \\
    \overset{\eqref{E:CE:proof3},\eqref{E:Ck,alpha}}&{\leq} C\bm_0^{1/2}\max\{\ell(J_H),\ell(J_L)\}^{1+\min\{\alpha,\gamma_{la}\}}.    
\end{aligned}
\end{equation*}
By construction, one can check that $|x_H - x_L|\geq c_0\max\{\ell(J_H),\ell(J_L)\}$, the last inequality and this consideration finish the proof of \eqref{I:P:CE:iii}. Again, it is a straightforward adaptation to prove the statement in \eqref{I:P:CE:iii} for $g^{\ast}_H$ and $g^{\ast}_L$.
\end{proof}

\section{Normal approximation on the external center manifold}\label{S:NA}

Having now the existence of the external center manifold, \cref{T:ECM}, we can construct the $\cM^\ast$-normal approximation (\cref{T:NA}) which is a $Q$-valued function defined on $\cM^\ast$ whose graph induces a current that approximates $T$ in cylinders defined by Whitney regions ($\bp^{-1}(\cL)$ for $\cL$ being a Whitney region). Let us make this notion precise below.

\begin{defi}[Whitney regions]
Let $\cM^\ast$ be the external center manifold relative to $\pi_0$ and $(\bGamma,\sW)$ the Whitney decomposition associated to it. Defining $\bPhi^\ast(x):= (x, \bphi^\ast(x))$, we call $\bPhi^\ast(\bGamma)$ the \emph{contact set}. Moreover, to each $L\in\sW$, we set the Whitney region $\cL$ on $\cM^\ast$ to be the following set $\cL:= \bPhi^\ast(J\cap [-\frac{7}{2},\frac{7}{2}]^m)$ where $J$ is the cube concentric to $L$ with side-length equal to $17\ell(L)/16$.
\end{defi}

\begin{defi} 
Let $\cM^\ast$ be an external manifold as in \cref{T:ECM}, we define
\begin{itemize}
    \item $\bU^\ast:= \left \{ x\in \R^{m+n}: \exists ! y = \bp^\ast(x)\in\cM^\ast, |x-y|<1, \text{ and } (x-y)\perp\cM^\ast \right\},$

    \item $\bp^\ast:\bU^\ast\to\cM^\ast$ is the map defined by the previous bullet,

    \item $\partial_l\bU^\ast := (\bp^\ast)^{-1}(\partial\cM^\ast)$ is the lateral boundary of $\cM^\ast$.
\end{itemize}
\end{defi}

\begin{assump}\label{A:NA}
Under Assumptions \ref{A:CM} and \ref{A:CM:parameters}, we further possibly decrease $\varepsilon_{cm}>0$ in order to have that $\bp^\ast\in C^{1,\alpha}(\overline{\bU^\ast})$ and $(\bp^\ast)^{-1}(y) = y + \overline{\tball{0}{1}{T_p\cM^\ast}}$ for every $y\in\cM^\ast$.
\end{assump}

We are now in position to prove the following corollary of the constructions made so far.

\begin{cor}\label{C:2.2}
Under \cref{A:CM}, \cref{A:CM:parameters}, and \cref{A:NA}, we have
\begin{enumerate}[\upshape (i)]
    \item\label{C:2.2:i} $\spt(\partial(T\res\bU^\ast))\subset \partial_l\bU^\ast$, $\spt(T\res [-\frac{7}{2},\frac{7}{2}]\times\R^n)\subset\bU^\ast$, and $\bp^\ast(T\res\bU^\ast) = Q\a{\cM^\ast}$;
    \item\label{C:2.2:ii} $\spt(\langle T, \bp^\ast, \bPhi^\ast(q)\rangle) \subset\left\{y:|\bPhi^\ast(q)-y| \leq C \bm_0^{1 / 2 m} \ell(L)^{1+\gamma_h}\right\}$ for every $q \in L \in \sW$, where $C=C\left(\gamma_e, \gamma_h, M_0, N_0, C_e, C_h\right)$;
    \item\label{C:2.2:iii} $\langle T, \bp^\ast, p\rangle=Q \a{ p }$ for every $p \in \bPhi^\ast(\bGamma)$.
\end{enumerate}
\end{cor}
\begin{rem}\label{R:U^star}
Notice that, the set $\bU^\ast$ is not necessarily contained in $\Sigma$ which is not an issue when paralleled to the approach in \cite[Corollary 2.2]{DS4}. Indeed, one can see that the set $\bU$ in \cite[Corollary 2.2]{DS4} is not guaranteed to be within $\Sigma$.
\begin{figure}[H]
    \centering
    \includegraphics{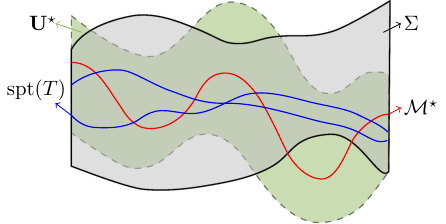}
\end{figure}
\end{rem}
\begin{proof}
The proof of the first and third affirmations in item \hyperref[C:2.2:i]{(i)}, items \hyperref[C:2.2:ii]{(ii)}, and \hyperref[C:2.2:iii]{(iii)}, are the same as presented in \cite[Corollary 2.2]{DS4}.  Let us prove $\spt(T\res [-\frac{7}{2},\frac{7}{2}]\times\R^n)\subset\bU^\ast$. Pick a point $p\in \spt(T\res [-\frac{7}{2},\frac{7}{2}]\times\R^n)$ and denote by $p^\prime$ the first $m$ coordinates of $\bp_{\pi_0}(p)$, then $p_0 := (p^\prime, \bphi^\ast(p^\prime))\in\cM^\ast$. If $p^{\prime}\in\bGamma$, it follows from item \hyperref[C:2.2:iii]{(iii)} that $p_0\in\bPhi^\ast(\bGamma)\cap\spt(T)$ and thus, by \eqref{L:1HB}, $| p - p_0| < 1$. If $p^\prime\in H$ where $H\in \sW^j$ for some $j$, we can proceed as follows:
\begin{equation*}
    |p-p_0| \leq |p - p_H| + |p_H - p_0| \overset{\textup{Prop. }\ref{P:tilting_opt_planes}\textup{ (iv)}}{\leq} 2^6 r_H + |p_H - p_0| = 2^6M_0\sqrt{m}\ell(L) + |p_H - p_0|.
\end{equation*}
Therefore, it remains to bound $|p_H - p_0|$. We show this bound below:
\begin{equation*}
\begin{aligned}
    |p_H - p_0| &\leq |p^\prime - x_H| + | \varphi^\ast_j(p^\prime) - y_H| + | \varphi^\ast_j(p^\prime) - \bphi^\ast(p^\prime)| \\
    \overset{p^\prime\in H}&{\leq} \sqrt{m}\ell(H) + | \varphi^\ast_j(p^\prime) - y_H| + | \varphi^\ast_j(p^\prime) - \bphi^\ast(p^\prime)| \\
    \overset{\textup{Thm. }\ref{T:ECM}\textup{ \hyperref[I:ECM:C0norm]{(A$^\ast$)}}}&{\leq} \sqrt{m}\ell(H) + C\bm_0^{1/2m} + |y_H| + | \varphi^\ast_j(p^\prime) - \bphi^\ast(p^\prime)| \\
    \overset{\textup{Prop. }\ref{P:tilting_opt_planes}\textup{ (iv)}}&{\leq} \sqrt{m}\ell(H) + C\bm_0^{1/2m} + C_h\bm_0^{1/2m}\ell(H) + | \varphi^\ast_j(p^\prime) - \bphi^\ast(p^\prime)| \\
    \overset{p_0\in\cM^\ast, \textup{Thm. }\ref{T:ECM}\textup{ \hyperref[I:ECM:convergence]{(A$^\ast$, D$^\ast$)}}}&{\leq} C\bm_0^{1/2m} < 1.
\end{aligned}
\end{equation*}
We proved that $|p - p_0|<1$ always holds true. Therefore, by \cref{A:NA}, we get that $p\in\bU^\ast$.
\end{proof}

A notion of Lipschitz approximation for nonlinear domains is given below with respect to $\cM^\ast$. Opportunely, we highlight the fact that the $\cM^\ast$-normal approximations are required to take values in $\Sigma$, see item (ii) in \cref{D:normalapprox}. Even though $F$ and $N$ will be defined on the external center manifold (which possibly lies outside $\Sigma$, in contrast with \cite{DS4}), their values are trapped in $\Sigma$ which will be important for computing inner and outer variations in what follows.  

\begin{defi}[$\cM^\ast$-normal approximation]\label{D:normalapprox}
An \emph{$\cM^\ast$-normal approximation of $T$} is given by a pair $(\cK, F)$ such that
\begin{enumerate}[\upshape (i)]
    \item $F:\cM^\ast\to\cA_Q(\bU^\ast)$ is Lipschitz w.r.t. the geodesic distance in $\cM^\ast$ and $F(p) = \sum_i \a{p+N_i(p)}$ where $N:\cM^\ast\to\cA_Q(\R^{m+n})$ is called \emph{the normal part of $F$},

    \item $N_i(p)\perp T_p\cM^\ast$ and $p+N_i(p)\in\Sigma$ for any $p\in\cM^\ast$ and $i$,

    \item $\cK\subset\cM^\ast$ is a closed set that contains $\bPhi^\ast(\bGamma\cap [-\frac{7}{2},\frac{7}{2}]^m)$ and $\bT_F\res (\bp^\ast)^{-1}(\cK) = T\res (\bp^\ast)^{-1}(\cK)$.
\end{enumerate}
\end{defi}
\begin{rem}
We point out that $N_i$ is not necessarily tangent to $\Sigma$. In fact, this is often the case, it can be explicitly seen in the proof of \cref{T:NA} when we construct the map $\Xi$. 
\end{rem}

We now state the existence and fine properties of an $\cM^\ast$-normal approximation. The strategy of the proofs of \cref{T:NA} and \cref{C:NA} follow \cite[Thm 2.4 and Cor 2.5]{DS4} and the final bounds are indeed the same. Nevertheless, given the difficulties of the $C^{2,\alpha}$ setting, some estimates have to be carefully carried out and the definition of the map $N$ is also subtler. In fact, the authors in \cite{DS4} rely on the fact that $\cM\subset\Sigma$ to construct a trivialization of class $C^{2,\alpha}$ (given that in their setting $\cM,\Sigma\in C^{3,\kappa}$) of the normal bundle of $\cM$. A crucial fact in their analysis is that the normal bundle of $\cM$ is a subspace of the tangent bundle of $\Sigma$ which is not guaranteed for the external center manifold.

\begin{thm}[Local estimates for the $\cM^\ast$-normal approximation]\label{T:NA}
Assume Assumptions \ref{A:CM}, \ref{A:CM:parameters}, and \ref{A:NA}, and let $\gamma_{na}:= \min\{\gamma_{la}/4,\alpha\}$. If $\varepsilon_{cm}=\varepsilon_{cm}(\gamma_e,\gamma_h, M_0,N_0, C_e, C_h)>0$ is sufficiently small, then there exists a constant $C=C\left(\gamma_e, \gamma_h, M_0, N_0, C_e, C_h\right)>0$ and an $\cM^\ast$-normal approximation $(\cK^\ast, F^\ast)$ such that, for every Whitney region $\cL$ associated to a cube $L \in \sW$, the following estimates hold true:
\begin{align}
    \Lip\left(\left.N^\ast\right|_{\cL}\right) \leq C \bm_0^{\gamma_{na}} \ell(L)^{\gamma_{na}}\text{ and }&\left\|\left.N^\ast\right|_{\cL}\right\|_{C^0} \leq C \bm_0^{1 / 2 m} \ell(L)^{1+\gamma_h} \label{E:T:NA:LipC0},\\
    \cH^m(\cL \setminus \cK^\ast)+ \|\bT_{F^\ast}-T\|\left((\bp^\ast)^{-1}(\cL)\right) &\leq C \bm_0^{1+\gamma_h} \ell(L)^{m+2+\gamma_h}, \\
    \int_{\cL}|D N^\ast|^2 &\leq C \bm_0 \ell(L)^{m+2-2\gamma_e}.
\end{align}
Moreover, for any $a>0$ and any Borel $\mathcal{V} \subset \cL$, we have
\begin{equation}\label{E:NA-control-average}
    \int_{\mathcal{V}}|\etaa \circ N^\ast| \leq C \bm_0\left(\ell(L)^{m+3+\gamma_e / 3}+a \ell(L)^{2+\gamma_{na}/ 2}|\mathcal{V}|\right)+\frac{C}{a} \int_{\mathcal{V}} \cG(N^\ast, Q \a{ \etaa \circ N^\ast })^{2+\gamma_{na}}.
\end{equation}
\end{thm}

\begin{cor}[Global estimates for the $\cM^\ast$-normal approximation]\label{C:NA}
Under the Assumptions of \cref{T:NA}, $N^\ast$ is the map from \cref{T:NA}, and denote $\cM^\ast_0 := \bPhi^\ast([-7/2,7/2]^m)$. Then there exists a constant $C=C\left(\gamma_e, \gamma_h, M_0, N_0, C_e, C_h\right)>0$ such that 
\begin{align}
    \Lip(N^\ast |_{\cM^\ast_0}) \leq C\bm_0^{\gamma_{na}}\text{ and }\|N |_{\cM^\ast_0}\|_{C^0} &\leq C\bm_0^{1/2m}\\
    \cH^m(\cM^\ast_0\setminus \cK^\ast) + \|\bT_{F^\ast}-T\|\left((\bp^\ast)^{-1}(\cM^\ast_0)\right) &\leq C \bm_0^{1+\gamma_h}\\
    \int_{\cM^\ast_0}|D N^\ast|^2 &\leq C \bm_0 .
\end{align}
\end{cor}

{We state now a simple linear algebra fact that will be used within the proof of the theorem.
\begin{lemma}\label{L:linear algebra}
Let $V$ and $W$ be vector spaces in $\R^N$ of dimension $i$ and $j$, respectively, with $i\leq j\leq N$ and denote $d_H(V,W)$ the Hausdorff distance between $V\cap\mathbb{S}^{N-1}$ and $W\cap\mathbb{S}^{N-1}$. There exists a positive dimensional constant $c_0$ such that, if $d_H(V, W) \leq c_0$, then $V^\perp + W = \R^N$.
\end{lemma}
\begin{proof}
Assume by way of contradiction that, for each $k\in\N$, there are $V_k$ and $W_k$ of dimension $i$ and $j$, respectively, such that $d_H(V_k,W_k)<1/k$ and the dimension of $V_k^\perp + W_k$ is strictly smaller than $N$. Thus, we have that $l_k:=\mathrm{dim}(V_k^\perp \cap W_k) > j-i\geq 0$. Take $\{\eta^k_1, \ldots, \eta^k_{j-i+1}, \ldots, \eta^k_{l_k}\}$ an orthonormal basis for $V_k^\perp \cap W_k$, $\{v^k_1,\ldots, v^k_i\}$ an orthonormal basis of $V_k$, and define $V_k^* := \mathrm{span}(v_1^k, \ldots, v_i^k, \eta^k_1, \ldots, \eta^k_{j-i})$ which has the same dimension as $W_k$. By definition of $V_k^*$, we have that $|\bp_{V_k^*} - \bp_{W_k}| = |\bp_{\mathrm{span}(\eta^k_{j-i+1}, \ldots,\eta^k_{l_k})}| \leq |\bp_{W_k} - \bp_{V_k\cap W_k}|$. Hence, we conclude the following:
\begin{equation*}
    \frac{1}{k} > d_H(V_k,W_k) \geq \sup_{x\in\oball{1}\cap W_k}|x - \bp_{V_k}x| = \sup_{x\in\oball{1}}|\bp_{W_k}x - \bp_{V_k\cap W_k}x| = |\bp_{W_k} - \bp_{V_k\cap W_k}| \geq |\bp_{V_k^*} - \bp_{W_k}|.
\end{equation*}
Taking $k$ to infinity, we denote $V$, $V^*$, and $W$ the limits of $V_k$, $V_k^*$, and $W_k$, respectively. By the displayed equation above, we have that $V^* = W$ which is in contradiction with the fact that $\eta^k_{j-i+1}$ converges to a nonzero vector $\eta_{j-i+1} \in W\cap (V^*)^\perp$.
\end{proof}}

\begin{proof}[Proof of \cref{T:NA}]
Running the same argument as in \cite[Subsection 6.2]{DS5} with $\cM^\ast\in C^{3,\kappa}$ (\cref{T:ECM}) in place of the center manifold $\cM$ which in our setting is only $C^{2,\alpha}$ (\cref{T:CM}), we can apply \cite[Theorem 5.1]{DS2} for $\cM^\ast$ to get, for each $L\in\sW^j$, $Q$-valued maps $N_L$ and $F_L$ satisfying the following properties:

\begin{itemize}
    \item $N_L:\cL^\prime\subset\cM^\ast\to \cA_Q(\R^{m+n})$ and $F_L:\cL^\prime\subset\cM^\ast\to \cA_Q(\bU^\ast)$, where $\cL^\prime := \bPhi^\ast(J)$ with $J$ being the cube concentric to $L$ with side-length $\ell(J) = 9\ell(L)/8$,

    \item $F_L(p) = \sum_i\a{p+ (N_L)_i(p)}$ and $(N_L)_i(p)\perp T_p\cM^\ast$ for every $p\in\cL^{\prime}$,

    \item and $\bG_{f_L}\res ((\bp^\ast)^{-1}(\cL^{\prime})) = \mathbf{T}_{F_L}\res ((\bp^\ast)^{-1}(\cL^{\prime}))$.
    
\end{itemize}

The functions $N_L$ and $F_L$ do not satisfy all the properties required in \cref{D:normalapprox}. For this reason, we have to extend their domains of definition and modify them in order to take values within $\Sigma$ and be orthogonal to $\cM^\ast$. Indeed, when we run the argument on \cite[page 537]{DS4}, we obtain, using the same notation of \cite{DS4}, functions $\hat{F}$ and $\hat{N}$ defined on the whole \emph{external} center manifold $\cM^\ast$ satisfying that $\hat{N}_i(p)\perp T_p\cM^\ast$ holds for every $p\in{\cM^\ast}$. Moreover, these functions also verify items (i) and (iii) of \cref{D:normalapprox}.

The next and last step is to modify $\hat{F}$ in order to ensure item (ii) of \cref{D:normalapprox}, i.e., that $p+N_i(p)\in\Sigma$ and $N_i(p)\perp T_p\cM^\ast$ for every $p\in\cM^\ast$. Such step has to be done cautiously, since the authors in \cite{DS4} use the fact that $\cM\subset\Sigma$, thus we perform the complete proof.

{Take $p\in\cM^\ast$ and $p^\prime = \mathfrak{P}(p)$ to be the nearest point to $p$ in $\Sigma$. Notice that, thanks to \Cref{A:nearest pt proj}, the nearest point projection to $\Sigma$ is $C^{1,\alpha}$ in $\oball{6\sqrt{m}}$ which together with \Cref{T:ECM} guarantee that $\cM^\ast$ is in the domain of $\mathfrak{P}$.}

{Since $\cM^\ast\in C^{3,\kappa}$, we can take an orthonormal $C^{2,\kappa}$-tangent bundle $\{t_1(p), \ldots, t_m(p)\}$ and also an orthonormal $C^{2,\kappa}$-normal bundle $\{\nu_1(p), \ldots, \nu_n(p)\}$ of $\cM^\ast$ at $p$, see for instance \cite[App. A]{DS2}.}

\begin{minipage}{\linewidth}
      \centering
      \begin{minipage}{0.45\linewidth}
          \begin{figure}[H]
              \includegraphics[width=\linewidth]{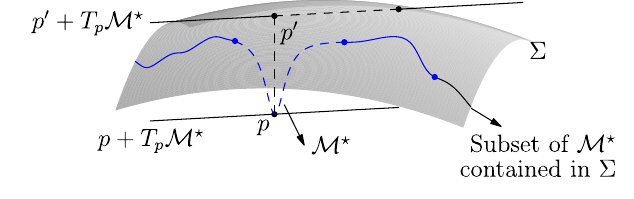}
          \end{figure}
      \end{minipage}
      \hspace{0.05\linewidth}
      \begin{minipage}{0.4\linewidth}
          \begin{figure}[H]
              \includegraphics[width=\linewidth]{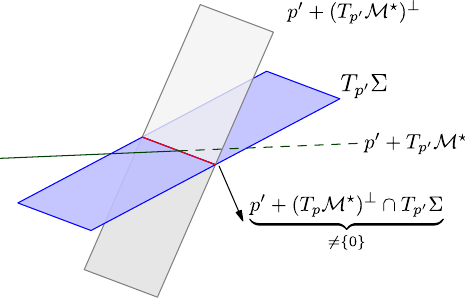}
          \end{figure}
      \end{minipage}
  \end{minipage}

{We now define $\varkappa_p:= (T_p\cM^\ast)^\perp\cap T_{p^\prime}\Sigma$ and claim that the dimension of $\varkappa_p$ is equal to $\bar{n}$ for any $p\in B_R$. Indeed, we define $\Phi(x) := (x,\bphi^\ast(x)) = p$ and recall that $T_p\cM^\ast = \mathrm{im}(D\Phi_x)$. Using $C^{1,\alpha}$-estimates on $\bphi^\ast$ at $0$ (namely \hyperref[I:ECM:HolderNorm]{(B$^\ast$)}) and \eqref{Sigma-close-to-pi_0}, we obtain that
\begin{equation*}
    d_H(T_p\cM^\ast, T_{p^\prime}\Sigma ) \leq d_H( T_p\cM^\ast, \pi_0) + d_H(T_{p^\prime}\Sigma, \pi_0) \leq C\bm_0^{1/2}.
\end{equation*}
So, as long as we (possibly) decrease $\varepsilon_{cm}>0$ a bit further, we can apply \Cref{L:linear algebra} to obtain that $(T_p\cM^\ast)^\perp + T_{p^\prime}\Sigma = \R^{m+n}$ and thus $\mathrm{dim}(\varkappa_p) = \bar{n}$.}

{Write $\Sigma\cap B_R =\{ G=0\}\cap B_R$ for some $G\in C^{2,\alpha} (\R^{m+n}, \R^{n-\bar{n}})$, it is clear that $T_{\mathfrak{P}(p)}\Sigma = \mathrm{Ker}(DG_{\mathfrak{P}(p)})$. We can then define the following map
\begin{equation*}   
\begin{aligned} 
    \mathcal{T} :\bigsqcup_{p\in\cM^\ast} \varkappa_p\times\R^{m+n} \to \R^{n-\bar{n}}\times \R^m \times \R^{\bar{n}} = \R^{m+n} \ \mbox{ given by }\\
    \mathcal{T}(p, v, \Xi) = \left(  G(\mathfrak{P}(p) + \Xi),   \langle \mathfrak{P}(p) + \Xi, t_1(p)\rangle, \ldots, \langle \mathfrak{P}(p) + \Xi, t_m(p) \rangle  , \bp_{\varkappa_p}(\Xi) - v\right). 
\end{aligned}\end{equation*}    
This map is $C^{1,\alpha}$, since it inherits the lowest regularity among the maps in its definition, namely $\mathfrak{P}$. We wish to apply the implicit function theorem to solve the system $\mathcal{T}(\cdot, \cdot, \cdot) = (0,0,0)$. To this end, we set the initial condition as follows:
\begin{equation*}   
\begin{aligned} 
    p_0 := (0,\bphi^\ast(0)), \ v_0 := -\bp_{\varkappa_p}(\mathfrak{P}(p_0)), \\ \mbox{ and } \Xi_0 := v_0 -\bp_{T_{p_0}\cM^\ast}(\mathfrak{P}(p_0)) - \bp_{(T_{p_0}\cM^\ast)^\perp\cap (T_{\mathfrak{P}(p_0)}\Sigma)^\perp}(\mathfrak{P}(p_0)).
\end{aligned}\end{equation*} 
Noticing that $\mathfrak{P}(p_0) + \Xi_0 = 0$ and recalling that the origin belongs to $\Sigma$, i.e., $G(0)=0$, it is a straightforward computation to check that $\mathcal{T}(p_0, v_0, \Xi_0) = (0,0,0)$ and also that the following holds 
\[ D_\Xi \mathcal{T}_{(p_0, v_0, \Xi_0)} = \begin{bmatrix}
    D G_{\mathfrak{P}(p_0) + \Xi_0} \\ t_1(p_0) \\ \ldots \\ t_m(p_0) \\ \bp_{\varkappa_p}
\end{bmatrix} = \begin{bmatrix}
    D G_0 \\ t_1(p_0) \\ \ldots \\ t_m(p_0) \\ \bp_{\varkappa_p}
\end{bmatrix},\]
where, by abuse of notation, we are using $\bp_{\varkappa_p}$ for the orthogonal projection and its associated matrix. We have then that 
\[
    v\in \mathrm{Ker}(D_\Xi \mathcal{T}_{(p_0, v_0, \Xi_0)})\Leftrightarrow v\in \varkappa_{p_0}^\perp\cap (T_{p_0}\cM^\ast)^\perp\cap \mathrm{Ker}(D G_0) = (T_{\mathfrak{P}(p_0)}\Sigma)^\perp\cap (T_{p_0}\cM^\ast)^\perp\cap T_0\Sigma .
\]
Now, since $\Sigma$ is regular and $c(\Sigma) \leq \varepsilon_{cm}$, we can guarantee that $d_H(T_0\Sigma, T_{\mathfrak{P}(p_0)}\Sigma) \leq c_0$ and applying \Cref{L:linear algebra}, we get that either $(T_{\mathfrak{P}(p_0)}\Sigma)^\perp + T_0\Sigma = \R^{m+n}$ or $T_{\mathfrak{P}(p_0)}\Sigma + (T_0\Sigma)^\perp = \R^{m+n}$. Both cases imply $(T_{\mathfrak{P}(p_0)}\Sigma)^\perp \cap T_0\Sigma = \{0\}$ and hence $\mathrm{Ker}(D_\Xi \mathcal{T}_{(p_0, v_0, \Xi_0)}) =\{0\}$. We are finally under the conditions to apply the implicit function theorem and then obtain $\Xi = \Xi(p,v)$ which is a $C^{1,\alpha}$-map and such that $\mathcal{T}(p,v,\Xi(p,v)) = (0,0,0)$, i.e., by definition of $\mathcal{T}$, it satisfies:
\begin{equation*}
    \Xi: \bigsqcup_{p\in\cM^\ast}\varkappa_p \to \R^{m+n}\text{ such that }\begin{cases}
        \bullet\quad p^{\prime} + \Xi(p,v)\text{ is the only point in }\Sigma\\
        \quad\quad\quad\quad\quad\text{ that is orthogonal to }T_p\cM^\ast,\\
        \bullet\quad \bp_{\varkappa_p}(\Xi(p,v)) = v.
    \end{cases}
\end{equation*}
}

We can now define the $\cM^\ast$-normal approximation (which does \emph{not} necessarily satisfy $N_i(p)\in T_{p^\prime}\Sigma$) as follows:
\begin{equation*}
    N: \cM^\ast \rightarrow \cA_Q(\R^{m+n})\text{ defined by }  \sum_{i=1}^Q\a{\Xi(p,\bp_{\varkappa_p}\left(\hat{N}_i(p)\right)}.
\end{equation*}

Since $p+\hat{N}_i(p)$ belongs to the support of $T$ and then to $\Sigma$, it is clear that $N(p) = \hat{N}(p)$ for any point in the good set $\cK^\ast$. To simplify the notation, denote $\Omega(p,v):= \Xi(p,\bp_{\varkappa_p}\left(v\right))$. As mentioned before, $\Omega$ is a $C^{1,\alpha}$ map. Therefore, we obtain 
\begin{equation}\label{6.8}
    |\Omega(p,v) - \Omega(p,u)| \leq C_0 |v-u|.
\end{equation}

Moreover, since $\Omega(p,0) = 0$ for every $p$, we have $D_p\Omega(p,0) = 0$, thus $|D_p\Omega(p,v)| \leq C_0 |v| ^{\alpha}$. Hence, we obtain that 
\begin{equation}\label{6.9}
    |\Omega(p,v) - \Omega(q,v)| \leq C_0 |v|^{\alpha} |p-q|.
\end{equation}

Fix two points $p,q\in\cL$ and assume that $\cG(\hat{N}(p), \hat{N}(q))^2 = \sum_i |\hat{N}_i(p) - \hat{N}_i(q)|^2$. We now have
\begin{equation*}
\begin{aligned}
    \cG(N(p), N(q))^2  &\leq 2\sum_i |\Omega(p, \hat{N}_i(p)) - \Omega(p,\hat{N}_i(q))|^2 + 2\sum_i |\Omega(p, \hat{N}_i(q)) - \Omega(q,\hat{N}_i(q))|^2 \\
    \overset{\eqref{6.8}, \eqref{6.9}}&{\leq} C_0 \cG(\hat{N}(p), \hat{N}(q))^2 + C \sum_i |\hat{N}_i(q)|^{2\alpha}|p-q|^2 \\
    &\leq C^2\bm_0^{2\gamma_{na}}\ell(L)^{2\gamma_{na}}|p-q|^2 + C^2\bm_0^{1/m}\ell(L)^{2\alpha(1+\gamma_h)}|p-q|^2. 
\end{aligned}
\end{equation*}

Thanks to the fact that $\gamma_{na}\leq \alpha$, the last inequality concludes the proof of \eqref{E:T:NA:LipC0}. The proof of the remaining estimates goes along the same lines as \cite[Thm 2.4]{DS4}. 
\end{proof}

\section{Frequency function and blow-up argument}\label{S:BU}

In this section, we show how to prove \cref{T:MAIN} using the approach in \cite{Alm} and \cite{DS5}. To this end, we set some notations and definitions. We start by setting the rescaling functions by $\iota_{p,r}(x) := \frac{x-p}{r}$ for any $p,x\in\R^{m+n}$ and $r>0$ and denoting the rescaled currents as $T_{p,r} := (\iota_{p,r})_{\sharp} T$.

\begin{assump}[Contradiction assumption]\label{A:FinalContradiction}
Assume that \cref{T:MAIN} is false. Precisely, assume that: there exists $m\geq 2, \bar{n}\geq 0, n\geq 1, \alpha>0, \Sigma$ and $T$ such that $\Sigma$ is an $C^{2,\alpha}$ embedded $(m+\bar{n})$-submanifold of $\R^{m+n}$, $T$ is an integral $m$-current in $\Sigma$ that minimizes area, and
\begin{equation*}
    \cH^{m-2+\tau}(\sing(T)) > 0\text{, for some }\tau >0.
\end{equation*}
\end{assump}

Under \cref{A:FinalContradiction}, we apply \cite[Prop 1.3]{DS5} to obtain the following contradiction sequence, which will allow us to derive a contradiction from \cref{A:FinalContradiction} and therefore conclude the proof of \cref{T:MAIN}.

\begin{lemma}[Contradiction sequence]\label{L:ContradictionSeq}
Assume \cref{A:FinalContradiction} and let $\bm_0 < \varepsilon_{bu}$ where $\varepsilon_{bu}\in (0, \varepsilon_{cm})$. There exist two real numbers $\tau, \eta>0$ and a sequence $r_k \downarrow 0$ such that:
\begin{enumerate}[\upshape (i)]
    \item $\Theta^m(T, 0 ) = Q$;

    \item\label{L:ContradictionSeq:Esmall} $\bE\left(T_{0, r_k}, \oball{6 \sqrt{m}}\right)$ goes to $0$ as $k$ goes to $+\infty$;

    \item $\lim _{k \to +\infty} \cH_{\infty}^{m-2+\tau}\left(\left\{p: \Theta^m(T_{0, r_k},p) = Q\right\} \cap \oball{1}\right)>\eta$; 

    \item $\cH^m\left(\left(\oball{1} \cap \spt\left(T_{0, r_k}\right)\right) \setminus \left\{p: \Theta^m(T_{0, r_k},p) = Q\right\}\right)>0$  for all $k \in \N$;

    \item $T_0\Sigma = \R^{m+\bar{n}}$, $\partial T\res \oball{6\sqrt{m}} = 0$;

    \item\label{E:2form small} $c(\Sigma\cap \oball{7\sqrt{m}})\leq\varepsilon_{bu}$;

    \item $\|T\|\left( \oball{6\sqrt{m}r}\right) \leq r^m \left( Q\omega_m (6\sqrt{m})^m + \varepsilon_{bu}\right)$ for any $r \in (0,1)$.
\end{enumerate}
\end{lemma}

We bring the readers' attention to the fact that the current $T$ of \cref{A:FinalContradiction} satisfies, thanks to \cref{L:ContradictionSeq}, all the requirements in \cref{A:CM} \emph{except} the smallness of the excess. However, \eqref{L:ContradictionSeq:Esmall} of \cref{L:ContradictionSeq} ensures that the excess of the rescalings of $T$ with respect to the sequence $\{r_k\}_k$ converges to $0$. This motivates us to define the \emph{intervals of flattening} that capture all radii $r\in(0,6\sqrt{m})$ for which the rescaled current $T_{0,r}$ is under \cref{A:FinalContradiction}. Thus, the external center manifold and the normal approximations can be constructed for each $T_{0,r}$.

\begin{defi}[Intervals of flattening and sequence of external center manifolds $\cM^\ast_j$]
Assume \cref{A:FinalContradiction}. We define the following set of radii:
\begin{equation*}
    \mathcal{R} := \left \{ r\in (0,1]: \bE( T_{0,r}, \oball{6\sqrt{m}}) \leq \varepsilon_{bu}^2\right\}.
\end{equation*}
If $\left\{s_k\right\} \subset \mathcal{R}$ and $s_k \uparrow s$, then $s \in \mathcal{R}$. Now, we cover $\mathcal{R}$ with a family of interval $\mathcal{F}=\left\{( s_j, t_j]\right\}_j$ called \emph{intervals of flattening}, defined as follows: $t_0:=\max \{t: t \in \mathcal{R}\}$, by induction, assume that $t_j$ is defined, and hence also $t_0>s_0 \geq t_1> s_1 \geq \ldots>s_{j-1} \geq t_j$. We also define the following objects:
\begin{itemize}
    \item $T_j:=T_{0, t_j}\res\oball{6 \sqrt{m}}, \Sigma_j:=\iota_{0, t_j}(\Sigma) \cap \oball{7 \sqrt{m}}$. Moreover, consider for each $j$ an orthonormal system of coordinates so that $\pi_0:= \R^m \times\{0\}$ is the optimal $m$-plane for the excess, i.e., $\bE\left(T_j, \oball{6 \sqrt{m}}, \pi_0\right)=\bE\left(T_j, \oball{6 \sqrt{m}}\right)$;

    \item Let $\cM^\ast_j$ be the external center manifold and $N^\ast_j$ the $\cM^\ast_j$-normal approximation constructed in \cref{T:ECM} and \cref{T:NA} applied to $T_j$ and $\Sigma_j$ with respect to the $m$-plane $\pi_0$. Notice that $T_j$ and $\Sigma_j$ are under \cref{A:CM}, \ref{A:CM:parameters}, and \ref{A:NA}, thanks to \cref{L:ContradictionSeq} and the definition of the intervals of flattening.
\end{itemize}
\end{defi}


With this, we consider the Whitney decomposition $\sW^{(j)}$ of $[-4,4]^m \subset \pi_0$ with respect to $T_j$, and we define
\begin{equation*}
    \quad s_j:=t_j \max \left(\left\{c_s^{-1} \ell(L): L \in \mathscr{W}^{(j)}\right.\right.\mbox{ and }\left.\left.c_s^{-1} \ell(L) \geq \operatorname{dist}(0, L)\right\} \cup\{0\}\right).
\end{equation*}
As in \cite{DS5}, one can see that $s_j / t_j<2^{-5}$ which then ensures that $(s_j, t_j]$ is a nontrivial interval. Next, if $s_j=0$, we stop the induction. Otherwise, we let $t_{j+1}$ be the largest element in $\mathcal{R} \cap\left[0, s_j\right]$ and repeat the procedure above.

\subsection{The frequency function}

\begin{defi}
Let $\phi: [0, +\infty)\to [0,1]$ be the piecewise linear function given by
\begin{equation*}
    \phi (r) := \begin{cases}
        1, &  r\in \left[0,\frac{1}{2}\right]\\
        2-2r, & r\in \left[\frac{1}{2}, 1\right] \\
        0, & r\in \left[1,+\infty) \right. .
    \end{cases}
\end{equation*}
Now we fix, for each $j$, $d_j$ to be the geodesic distance in $\cM_j^\ast$ between $p$ and $\bPhi^\ast_j(0)$. Then, we define
\begin{itemize}
    \item $\bD_j(r) := \int_{\cM^\ast_j}\phi\left( \frac{d_j(p)}{r}\right)|D N^\ast_j|^2(p)\mathrm{d}\cH^m(p)$;

    \item $\bH_j (r) := - \int_{\cM^\ast_j}\phi^\prime\left(\frac{d_j(p)}{r}\right)\frac{|N^\ast_j|^2(p)}{d_j(p)} \mathrm{d}\cH^m(p)$.
\end{itemize}
The \emph{frequency function} is then defined as $\bI_j (r) := \frac{r\bD_j(r)}{\bH_j(r)}$ whenever $\bH_j(r) >0$.
\end{defi}

\begin{rem}\label{R:inSigma}
A main ingredient to prove \cref{T:MAIN} is the monotonicity of the frequency function. To prove this, the authors in \cite[Thm 3.2]{DS5} take the inner and outer variations of $\bI_j$ and prove fine estimates for them. A crucial information in this analysis is the stationarity of $T$, more precisely, that the first variation of $T$ vanishes with respect to vector fields \emph{tangent} to $\Sigma$. To use this fact, we need to make sure that our normal approximations are supported in $\Sigma$, which we ensure in \cref{T:NA}.
Moreover, to carry out the computations in \cite{DS5}, it is also essential to have that $\spt(T)$ is a subset of $\bU$ (i.e., \cite[Corollary 2.2]{DS5}) which is also guaranteed in our setting with $\bU^\ast$ as we have shown in \Cref{C:2.2}, see also \Cref{R:U^star}.
\end{rem}
\begin{rem}\label{R:MeanCurvatureHolder}
We mention that one of the main reasons the authors in \cite{DS5} need the $C^{3,\alpha}$-regularity for the center manifold is to utilize the H\"{o}lder continuity of $D H_{\cM}$ in proving the monotonicity of the frequency function, where $H_{\cM}$ denotes the mean curvature of $\cM$. In the present setting, however, $\cM$ only possesses $C^{2,\alpha}$-regularity, which precludes us from using the derivative of $H_{\cM}$. Nevertheless, we have constructed $\cM^\ast$ so that normal approximations are defined over $\cM^\ast$, which has $C^{3,\alpha}$-regularity. This ensures that $D H_{\cM^\ast}$ is well-defined and exhibits H\"{o}lder continuity.
\end{rem}

Taking into consideration \cref{R:inSigma} and \cref{R:MeanCurvatureHolder}, we are therefore in position to apply the same machinery as in \cite{DS5} to obtain the monotonicity of the frequency function, i.e., \cref{T:FF}.

\begin{thm}[Main frequency estimate]\label{T:FF}
Provided $\varepsilon_{bu}$ is chosen small enough, there exists a geometric constanct $C_0 > 0$ such that 
\begin{equation*}
    \bI_j(a) \leq C_0 ( 1 + \bI_j(b))\text{, for every }[a,b]\subset \left[\frac{s_j}{t_j},3\right]\text{ such that }\bH_j|_{[a,b]} >0.
\end{equation*}
\end{thm}

\subsection{Blow-up argument}

We now show how to obtain a function at the limit that provides us the desired contradiction to conclude \Cref{T:MAIN}. The proof of \cref{T:FinalBU} resembles that of \cite[Thm 6.2]{DS5}, however, it has to be subtly changed when dealing with the external center manifolds $\cM^\ast_j$ and the $\cM^\ast_j$-normal approximations $N^\ast_j$ since the arguments in \cite{DS5} strongly rely on the fact that $\cM\subset\Sigma$ and the $C^{3,\alpha}$-regularity of $\Sigma$ and $\cM$. 

\begin{thm}[Final blow-up]\label{T:FinalBU}
The maps $N_k^{\ast,b}$ strongly converge, up to subsequences, in $L^2(B_{3/2})$ to a function $N_\infty^{\ast,b}\in W^{1,2}(B_{3/2}, \cA_Q(\{0\}\times \R^{\bar{n}} \times \{0\}))$ 
that satisfies the following:
\begin{enumerate}[\upshape (i)]
    \item $N_{\infty}^{\ast,b}$ is a minimizer of the Dirichlet energy in $B_t$ for any $t\in (5/3,3/2)$;

    \item $\|N_{\infty}^{\ast,b}\|_{L^2(B_{3/2})} = 1$;

    \item $\etaa\circ N_\infty^{\ast,b} \equiv 0$.
\end{enumerate}
\end{thm}

\begin{proof}[Proof of \cref{T:MAIN}]
The approach implemented in \cite[Pf of Thm 0.3]{DS5} can be carried out along the very same lines to get \cref{T:MAIN} from the above properties of the multivalued limit function $N_{\infty}^{\ast,b}$. For this reason, we omit the proof here.
\end{proof}

To prove \Cref{T:FinalBU}, we need to introduce some notation as follows. Consider a current $T$ satisfying Assumption \ref{A:FinalContradiction}. As in \cite[Proposition 2.2]{DS5}, for each radius \( r_k \) produced by Lemma \ref{L:ContradictionSeq}, there is an interval of flattening such that \( r_k\in (s_{j(k)}, t_{j(k)}] \). We {define \( \bar{s}_k \) to be such that $ \bar{s}_k/t_{j(k)}$ is the radius} given by the Reverse Sobolev inequality \cite[Corollary 5.3]{DS5} applied to \( r = r_k{/t_{j(k)}} \). {Notice that \(3r_k/2 \leq \bar{s}_k \leq 3r_k \)}. We then put \( \bar{r}_k := {\frac{2\bar{s}_k}{3t_{j(k)}}} \) and rescale and translate our objects as below:
\begin{itemize}
    \item \( \bar{T}_k:= (\iota_{0,\bar{r}_k})_\sharp T_{j(k)} = ((\iota_{0,\bar{r}_k t_{j(k)}})) T)\res\mathbf{B}_{\frac{6\sqrt{m}}{\bar{r}_k}}, \overline{\Sigma}_k:= \iota_{0,\bar{r}_k}(\Sigma_{j(k)}) \) and \( \overline{\cM}^\ast_k := \iota_{0,\bar{r}_k}(\cM^\ast_{j(k)})\),
    
    \item \( \overline{N}^\ast_k : \overline{\cM}^\ast_k \to \R^{m+n} \) are the rescaled \( \overline{\cM}^\ast_k \)-normal approximations given by
    \[
    \overline{N}^\ast_k(p) = \bar{r}_k^{-1} N^\ast_{j(k)}(\bar{r}_k p).
    \]
\end{itemize}
We can assume \( T_0\Sigma = \R^{m+\bar{n}} \times \{0\} \), thus the ambient manifolds \( \overline{\Sigma}_k \) converge to \( \R^{m+\bar{n}} \times \{0\} \) locally in \( C^{2,\alpha} \). Furthermore, since \( \frac12 < \frac{r_k}{\bar{r}_kt_{j(k)}} < 1 \), Lemma \ref{L:ContradictionSeq} implies  that
\begin{equation*}
    \bE(\bar{T}_k, \oball{1/2}) \leq C \bE(T, \oball{r_k}) \to 0.
\end{equation*}
Without loss of generality, we can assume that $\bar{T}_k$  locally converges to $Q\a{\pi_0}$. Moreover, from Lemma \ref{L:ContradictionSeq}, it follows that
\begin{equation}
    \cH_\infty^{m+2+\tau}(D_Q(\bar{T}_k) \setminus\mathbf{B}_1) \leq C_0 r_k^{-(m+2+\tau)}\cH_\infty^{m+2+\tau}(D_Q(T) \setminus\mathbf{B}_{r_k})\ge\eta>0,
\end{equation}
where $C_0$ is a geometric constant. As in \cite[Lemma 6.1]{DS5}, we can show that $\overline{\cM}^\ast_k$ locally converge in $C^{3,\kappa/2}$ to the $m$-plane $\pi_0$. We thus define the blow-up maps $N_k^{\ast, b}:\mathrm{B}_3\subset \pi_0=\R^m\to\cA_Q(\R^{m+n})$:
\begin{equation}
  N_k^{\ast, b}(x):=\frac{\bar{N}^\ast(\mathbf{e}_k(x))}{\bh^\ast_k}, \ \bh^\ast_k:=\left\|\bar{N}^\ast_k\right\|_{L^2\left(\mathcal{B}_{3/2}\right)}, \mbox{ and }\bar{p}^\ast_k:=\frac{\mathbf{\Phi}^\ast_{j(k)}(0)}{\bar{r}_k},
\end{equation}
where $\mathbf{e}_k:=exp_{\bar{p}^\ast_k}$ denotes the exponential map of $\overline{\cM}^\ast_k$ at $\overline{p}_k^\ast$. Henceforth, we assume, without loss of generality, that we have applied a suitable rotation to each $\bar{T}_k$ so that the tangent plane $T_{\bar{p}^\ast_k}\overline{\cM}^\ast_k$ coincides with $\R^m\times\{0\}$.

To prove \Cref{T:FinalBU}, we will proceed as follows. We approximate, by smoothing, both $\overline{\cM}^\ast_k$ and $\overline{\Sigma}_k$ with suitable $C^{\infty}$-submanifolds of the ambient space $\R^{m+n}$ which we denote by $\overline{\cM}^\ast_{k,\varepsilon}$ and $\overline{\Sigma}_{k,\varepsilon}$, respectively. Using these smoothened approximations, we have the existence of a well-defined projection $\mathfrak{P}_{k,\varepsilon} = \mathfrak{P}_{\overline{\Sigma}_{k,\varepsilon},\overline{\cM}^\ast_{k,\varepsilon}}$ satisfying:
\begin{equation*}
    \mathfrak{P}_{k,\varepsilon}(\overline{\cM}^\ast_{k,\varepsilon})  \subseteq\overline{\Sigma}_{k,\varepsilon}\mbox{ and }\mathfrak{P}_{k,\varepsilon}\mbox{ is of class }C^{\infty}(\mathbf{B}_{6\sqrt{m}}).
\end{equation*} 
From this, we can run similar constructions to those in \cite[Section 7.2]{DS5} to the pair $\left(\overline{\cM}^\ast_{k,\varepsilon},\overline{\Sigma}_{k,\varepsilon}\right)$ where we need to factorize the normal approximation $N^\ast_j$ through the projection above and let $\varepsilon\to 0$ and $k\to+\infty$ to conclude the proof. The subtleties of controlling all the error terms involved with this smoothing procedure are detailed below.

\begin{proof}[Proof of \Cref{T:FinalBU}]

We divide the proof into some steps.

\emph{\underline{Step 1}: Proof of the existence of $N_\infty^{\ast,b}\in W^{1,2}(B_{3/2}, \cA_Q(\{0\}\times \R^{\bar{n}} \times \{0\}))$, (ii), (iii), and $L^2$-strong convergence.}

Without loss of generality, we translate the manifolds $\overline{\cM}^\ast_k$ so that the rescaled points $\overline{p}^\ast_k=\bar{r}_k^{-1} \bPhi^\ast_{j(k)}(0)$ all coincide with the origin $0_{\R^{m+n}}$. Let us define 
\begin{equation*}
    \overline{F}^\ast_k: \mathcal{B}^\ast_{3/2} \subset \overline{\cM}^\ast_k \rightarrow \cA_Q\left(\R^{m+n}\right) \mbox{ and }\overline{F}^\ast_k(x):=\sum_i \a{ x+\left(\overline{N}_k^\ast\right)_i(x) }.
\end{equation*} 
To simplify the notation, set ${\bp}_k^\ast:=\bp_{\overline{\cM}^\ast_k}$.
We start by showing the existence of a suitable exponent $\gamma>0$ such that
\begin{align}
    \operatorname{Lip}\left(\left.\overline{N}_k^\ast\right|_{\mathcal{B}^\ast_{3 / 2}}\right) & \leq  C{\bh_k^\ast}^{\gamma_{na}}, \quad \left\|\overline{N}_k^\ast\right\|_{C^0\left(\mathcal{B}^\ast_{3 / 2}\right)} \leq C\left({\bm_{0,j(k)}} \bar{r}_k\right)^{\gamma_{na}}, \label{Eq:(7.1)} \\
    \mathbf{M}\left((\mathbf{T}_{\overline{F}^\ast_{ k }}-\overline{T}_{k })\res{\bp^\ast_k}^{-1}\left(\mathcal{B}^\ast_{3/2}\right)\right) & \leq  C {\bh_k^\ast}^{2+2 \gamma_{na}}, \mbox{ and }
    \int_{\mathcal{B}_{3/2}}\left|\boldsymbol{\eta} \circ \overline{N}_k^\ast\right| \leq  C {\bh_k^\ast}^2.\label{Eq:(7.3)}
\end{align}
Using the fact that $3 \bar{r}_k/2 \in\left(s_{j(k)}/t_{j(k)}, 3\right)$ and \eqref{E:NA-control-average} with $a=\bar{r}_k$, we infer as in \cite[Section 7.1]{DS5} that
\begin{equation*}
\begin{aligned}
    \left\|N^\ast_{j(k)}\right\|_{C^0\left(\mathcal{B}^\ast_{3 \bar{r}_k/2}\left(p^\ast_{j(k)}\right)\right)} \leq C\bm_{0,j(k)}^{1/2m} \bar{r}_k^{1+\gamma_h}, \ \operatorname{Lip}&\left(\left.N^\ast_{j(k)}\right|_{\mathcal{B}^\ast_{3\bar{r}_k/2}\left(p_{j(k)}\right)}\right) \leq C\bm_{0,j(k)}^{\gamma_2} \max _i \ell_i^{\gamma_2}, \\
    \mathbf{M}\left(\left(\mathbf{T}_{F^\ast_{j(k)}}-T_{j(k)}\right)\res {\bp^\ast_k}^{-1}\left(\mathcal{B}^\ast_{3\bar{r}_k/2}\left(p^\ast_{j(k)}\right)\right)\right) &\leq \sum_i\bm_{0,j(k)}^{1+\gamma_2} \ell_i^{m+2+\gamma_2}, \\
    \int_{\mathcal{B}^\ast_{3\bar{r}_k/2}\left(p^\ast_{j(k)}\right)}\left|\boldsymbol{\eta} \circ N^\ast_{j(k)}\right| \leq C {\bm_{0,j(k)}} \bar{r}_k \sum_i \ell_i^{2+m+\gamma_{la} / 2}&+\frac{C}{\bar{r}_k} \int_{\mathcal{B}^\ast_{3\bar{r}_k/2}\left(p^\ast_{j(k)}\right)}\left|N^\ast_{j(k)}\right|^2.
\end{aligned}
\end{equation*}
Arguing similarly to \cite[(4.12), (4.13), and (4.14)]{DS5} and using the Reverse Sobolev inequality proved in \cite[Corollary 5.3]{DS5}, we see that
\begin{equation}
\begin{aligned}\label{Eq:(7.4)}
    \sum_i {\bm_{0,j(k)}} \ell_i^{m+2+\frac{\gamma_2}{4}} & \leq C_0 \int_{\mathcal{B}^\ast_{3\bar{r}_k/2}\left(p^\ast_{j(k)}\right)}\left(\left|D N^\ast_{j(k)}\right|^2+\left|N^\ast_{j(k)}\right|^2\right)  \leq \frac{C_T}{\bar{r}_k^{2}} \int_{\mathcal{B}^\ast_{3\bar{r}_k/2}\left(p^\ast_{j(k)}\right)}\left|N^\ast _{j(k)}\right|^2,
\end{aligned}
\end{equation}
from which \eqref{Eq:(7.1)} and \eqref{Eq:(7.3)} follow by a simple rescaling. The constant $C_T$ on the right-hand side of \eqref{Eq:(7.4)} depends on $T$ but not on $k$. It is a consequence of these bounds and the Sobolev embedding (cf. \cite[Prop. 2.11]{DS1}) that the sequence $\{N_k^{\ast,b}\}_k$ weakly converges in $ W^{1,2}\left(B_{3/2},\cA_Q\left(\{0\}\times\R^{\bar{n}}\times\{0\}\right)\right)$ (in the sense of \cite[Definition 2.9]{DS1}) to a $Q$-valued function $N_{\infty}^{\ast,b}\in W^{1,2}\left(B_{3/2},\cA_Q\left(\R^{m+n}\right)\right)$. From this convergence and \eqref{Eq:(7.3)}, we derive that
\begin{equation*}
    \int_{B_{3/2}}\left|\boldsymbol{\eta} \circ N_{\infty}^{\ast,b}\right|=\lim _{k \rightarrow+\infty} \int_{B_{3/2}}\left|\boldsymbol{\eta} \circ N_k^{\ast,b}\right| \leq C \lim _{k \rightarrow+\infty} {\bh_k^\ast}=0 .    
\end{equation*}
We now check that $N_{\infty}^{\ast,b}$ must take its values in $\cA_Q\left(\{0\} \times \R^{\bar{n}} \times\{0\}\right)$. Consider the tangential part of ${\overline{N}_k^\ast}$ as in 
\begin{equation*}
    {\overline{N}_k^{\ast,T}}(x):=\sum_i \a{ \bp_{T_x \overline{\Sigma}_k}\left(\overline{N}_k^\ast)_i(x)\right) }.
\end{equation*}
One can verify that $\cG\left({\overline{N}_k^\ast}, {\overline{N}_k^{\ast,T}}\right)=\cG\left(Q\llbracket0\rrbracket, {\overline{N}_k^{\ast,\perp}}\right)\le C_0\left|{\overline{N}_k^\ast}\right|^2$ which leads to
\begin{equation*}
    \int_{B_{3 / 2}} \cG\left(N_k^{\ast,b}, {\bh_k^\ast}^{-1} {\overline{N}_k^{\ast,T}} \circ \mathbf{e}_k\right)^2  \leq C_0 {\bh_k^\ast}^{-2} \int_{\mathcal{B}^\ast_{3 / 2}}\left|{\overline{N}_k^\ast}\right|^4 
    \overset{\eqref{Eq:(7.1)}}{\leq} C\left({\bm_{0,j(k)}} \bar{r}_k\right)^{2 \gamma_{na}} \rightarrow 0 \mbox { as } k \rightarrow+\infty.
\end{equation*}
By the $C^{2,\alpha}$-convergence of $\overline{\Sigma}_k$ to $\R^{m+\bar{n}} \times\{0\}$, we conclude Step 1.

\emph{\underline{Step 2}: A suitable trivialization of the normal bundle through smoothing estimates and $W^{1,2}$-convergence.}

The difficulties we face in exploiting the arguments of \cite[Section 7.2]{DS5} are twofold: (1) the first is that we have to deal with an external center manifold $\overline{\cM}^\ast_k$ not lying inside $\overline{\Sigma}_k$, which introduce error terms that we have to account for, and (2) the fact that $\overline{\Sigma}_k$ is merely of class $C^{2,\alpha}$, which we deal with by approximating with $\overline{\Sigma}_{k,\varepsilon}$ of class $C^{\infty}$ obtained from $\Sigma$ first by rescaling by a factor $t_{j(k)}$, then mollifying it, and finally rescaling it again by a factor $\bar{r}_k$. Precisely, $\overline{\Sigma}_{k,\varepsilon}$ is obtained as follows. We introduce a parameter $\varepsilon>0$, mollify the functions $\mathbf{\Psi}_{j(k)}$ to get the functions $\mathbf{\Psi}_{k, \varepsilon}$ of class $C^{\infty}$, and then define
\begin{equation*}
    \Sigma_{k,\varepsilon}:=\operatorname{graph}(\mathbf{\Psi}_{j(k),\varepsilon})\mbox{ and }\overline{\Sigma}_{k,\varepsilon}:=\iota_{{\mathbf{\Psi}_{j(k),\varepsilon}(0)}, \bar{r}_k}\left(\Sigma_{j(k),\varepsilon}\right){+\mathbf{\Psi}_{j(k),\varepsilon}(0)}.
\end{equation*}

Furthermore, since $\cM^\ast_k$ is of class $C^{3,\kappa}$ the projection $\mathfrak{P}_k$ is only $C^{2,\kappa}$, we will also need to mollify $\cM^\ast_k$. Precisely, we define $\bPhi_{k,\varepsilon}\in C^{\infty}$ to be the mollified function obtained from $\bPhi_{j(k)}$,
\begin{equation*}
    \cM^\ast _{k,\varepsilon}:=\operatorname{graph}(\bPhi_{j(k),\varepsilon}), \ \overline{\cM}^\ast_{k,\varepsilon}:=\iota_{{\mathbf{\Psi}_{j(k),\varepsilon}(0)}, \bar{r}_k}\left(\cM^\ast _{k,\varepsilon}\right){+\mathbf{\Psi}_{j(k),\varepsilon}(0)}.
\end{equation*}

By construction, $\overline{\Sigma}_{k,\varepsilon}$ is approaching $\overline{\Sigma}_k$ in the $C^{2,\beta}$-topology for every $0<\beta<\alpha$ and $\overline{\cM}^\ast _{k,\varepsilon}\to\overline{\cM}^\ast_k$ in $C^{3,\beta_1}$ topology for every $0<\beta_1<\kappa$, when $\varepsilon\to 0^+$ for any fixed $k\in\N$. 

{As it is well known (see for instance \cite[Theorem 3.4]{AamariKimEtAl2019}) for every submanifold $M \hookrightarrow\R^{m+n}$ of class at least $C^{k,\alpha}$, we have a well-defined nearest point projection $\mathfrak{P}_{M}:U_s(M)\subseteq\R^{m+n}\to M$ of class $C^{k-1,\alpha}$ for every $s\in]0, s_0)$, where $s_0:=\frac{C(m,n)}{c(M)}$, $c(M)$ is the $C^0$-norm of the second fundamental form of $M$, and $U_s(M)$ is a tubular neighborhood of $M$ with radius $s$. }

{By construction of the approximating submanifolds and Assumption \ref{A:CM}, there exist $\mathbb{N}\ni k_0=k_0(m,n, \varepsilon_{cm})$, $0<\varepsilon_0=\varepsilon_0(m,n, \varepsilon_{cm})$ such that
$$\frac{C(m,n)c(\Sigma)}{2}\le c(\Sigma_{k,\varepsilon})\le2C(m,n)c(\Sigma),\:\forall k\ge k_0,\forall \varepsilon\in (0, \varepsilon_0).$$}

{Moreover, $\overline{\Sigma}_k$ and $\overline{\cM}^\ast_k$ are even flatter than $\Sigma$ and $\cM^\ast$, i.e., $0\le c(\Sigma_k)\le c(\Sigma)$ and $0\le c(\cM^\ast_k)\le c(\cM^\ast)$. As in Remark \ref{A:nearest pt proj}, this ensures that, if $k\in\N$ is larger and $\varepsilon>0$ is smaller than a geometric constant, we obtain $\mathfrak{P}_{k,\varepsilon}$ is a map of rank equal to the dimension of $\Sigma$ in a fixed neighborhood of the origin (e.g., $\mathbf{B}_{6\sqrt{m}}$), where we define $\mathfrak{P}_{k,\varepsilon}:U_{s_0}(\Sigma_{k,\varepsilon})\to\Sigma_{k,\varepsilon}$ to be the nearest point projection onto $\Sigma_{k,\varepsilon}$.}

{Notice also that $D_p\mathfrak{P}_{k,\varepsilon}(v+(T_{\mathfrak{P}_{k,\varepsilon}(p)}\overline{\Sigma}_{k,\varepsilon})^\perp)=0$ for $v:=\mathfrak{P}_{k,\varepsilon}(p)-p$. In fact, this follows from the fact that $\mathfrak{P}_{k,\varepsilon}$ is a diffeomorphism when restricted to the image of the normal exponential map of $\overline{\Sigma}_{k,\varepsilon}$. Furthermore, by convergence in $C^{1,\alpha}$ topology of the relevant tangent bundles, we know that, for every sequence $\varepsilon_k\downarrow0$ and for every $p\in\overline{\cM}^\ast_{k,\varepsilon_k}\cap\oball{6\sqrt{m}}$, we have that
$$|T_p\overline{\cM}^\ast_{k,\varepsilon_k}-T_{\mathfrak{P}_{k,\varepsilon_k}(p)}\mathfrak{P}_{k,\varepsilon_k}(\overline{\cM}^\ast_{k,\varepsilon_k})|\to0,\: k\to+\infty.$$
Hence, it can be taken arbitrarily small provided $k$ is large and $\varepsilon$ is small and so in particular smaller than $c_0$ of Lemma \ref{L:linear algebra}. }

{Denote $\widetilde{\cM}_{k,\varepsilon}:=\mathfrak{P}_{k,\varepsilon}\left(\overline{\cM}^\ast _{k,\varepsilon}\right)\subset\overline{\Sigma}_{k,\varepsilon}$ and notice that $T\widetilde{\cM}_{k,\varepsilon}$ is a sub-vector bundle of $T\overline{\Sigma}_{k,\varepsilon}$. Now, applying Lemma \ref{L:linear algebra} with $V=T_p\overline{\cM}^\ast_{k,\varepsilon_k}$ and $W=T_{\mathfrak{P}_{k,\varepsilon_k}(p)}\widetilde{\cM}_{k,\varepsilon}$, we obtain that $\mathfrak{P}_{k,\varepsilon}$ is a diffeomorphism from $\overline{\cM}_{k,\varepsilon}$ onto $\widetilde{\cM}_{k,\varepsilon}$ and we can construct an orthonormal frame which we denote by $\{\nu_1^{k,\varepsilon},\ldots, \nu_{\bar{n}}^{k,\varepsilon}, \varpi_1^{k,\varepsilon}, \ldots, \varpi_l^{k,\varepsilon}\}$ of $(T \overline{\cM}^\ast _{k,\varepsilon})^{\perp}$ satisfying
\begin{equation*}
\begin{aligned}
   \nu_1^{k,\varepsilon}(p), \ldots, \nu_{\bar{n}}^{k,\varepsilon}(p), \varpi_1^{k,\varepsilon}(p), \ldots, \varpi_l^{k,\varepsilon}(p)\mbox{ is orthonormal  at each point }p\in\overline{\cM}^\ast _{k,\varepsilon},\\
   \nu_j^{k,\varepsilon}(p) \in T_{\mathfrak{P}_{k,\varepsilon}(p)}\overline{\Sigma}_{k,\varepsilon}, \mbox{ and }\varpi_j^{k,\varepsilon}(p) \perp T_{\mathfrak{P}_{k,\varepsilon}(p)} \overline{\Sigma}_{k,\varepsilon}.
\end{aligned}
\end{equation*}}

{Arguing exactly as in \cite[Lemma 6.1]{DS5} with the pair $(\overline{\Sigma}_{k,\varepsilon}, \widetilde{\cM}_{k,\varepsilon})$, we get, for every fixed $\varepsilon>0$, that 
\begin{equation}\label{Eq:BeforeMollifyThenRescale1}
\tilde{\nu}_j^{k,\varepsilon} \rightarrow e_{m+j}^\varepsilon \quad \text { and } \quad \varpi_j^{k,\varepsilon} \rightarrow e_{m+\bar{n}+j}^\varepsilon \quad \text { in } C^{2, \kappa / 2}\left(\widetilde{\cM}_{k,\varepsilon}\right) \text { as } k\to\infty,
\end{equation}
where $\tilde{\nu}_j^{k,\varepsilon}=D\mathfrak{P}_{k,\varepsilon}(\nu_j^{k,\varepsilon})$ and $e_1^\varepsilon, \ldots, e_{m+\bar{n}+l}^\varepsilon$ is a orthonormal basis of $V_\varepsilon\oplus V_\varepsilon^\perp\equiv\R^{m+n}$ where $V_\varepsilon$ is the vector space of dimension $m+\bar{n}$ which is the limit, with respect to $k$, of the sequence $\overline{\Sigma}_{k,\varepsilon}$. Since $e_1^\varepsilon, \ldots, e_{m+\bar{n}+l}^\varepsilon$ is a constant orthonormal basis converging to the canonical base of $\R^{m+n}$ as $\varepsilon\to0$, we can deduce by a standard diagonalization process that there exists a sequence $\varepsilon_k\downarrow0$ such that 
\begin{equation}\label{Eq:BeforeMollifyThenRescale2}
\tilde{\nu}_j^{k,\varepsilon_k} \rightarrow e_{m+j} \quad \text { and } \quad \varpi_j^{k,\varepsilon_k} \rightarrow e_{m+\bar{n}+j} \quad \text { in } C^{2, \kappa / 2}\left(\widetilde{\cM}_{k,\varepsilon_k}\right) \text { as } k\to\infty.
\end{equation}}
{Finally, we obtain conclude
\begin{equation}\label{Eq:BeforeMollifyThenRescale}
\nu_j^{k,\varepsilon_k} \rightarrow e_{m+j} \quad \text { and } \quad \varpi_j^{k,\varepsilon_k} \rightarrow e_{m+\bar{n}+j} \quad \text { in } C^{2, \kappa / 2}\left(\overline{\cM}^\ast _{k,\varepsilon_k}\right) \text { as } k\to\infty,\footnote{In fact, they converge in $C^\infty$ topology.}
\end{equation}
where $e_1, \ldots, e_{m+\bar{n}+l}$ is the standard basis of $\R^{m+\bar{n}+l}=\R^{m+n}$. Henceforth, we will always consider sequences $\varepsilon_k$ satisfying \eqref{Eq:BeforeMollifyThenRescale}. Therefore, to avoid cumbersome notation, we will drop the $k$ from $\varepsilon_k$ in what follows.}

Now that \eqref{Eq:BeforeMollifyThenRescale} is established, it is easy to check that we can find $\delta>0$ (independent of $k$ and $\varepsilon$) such that, for $k\ge k_0(\delta)\in\mathbb{N}$ large enough and $0<\varepsilon< \varepsilon_0(\delta, k)$ small enough, there is a map $\psi^\ast_{k,\varepsilon}: \overline{\cM}^\ast _{k,\varepsilon} \times \R^{\bar{n}} \rightarrow \R^l$ converging to $0$ in $C^{2, \kappa / 2}$ (uniformly bounded in $C^{2, \beta_1}$ with $\frac\kappa2<\beta_1<\kappa\le\alpha$) given by the following property:
\begin{equation*}
    \mathfrak{P}_{{k,\varepsilon}}(p)+v \in \overline{\Sigma}_{k,\varepsilon}\Longleftrightarrow v^{\perp}=\psi^\ast_{k,\varepsilon}\left(p, v^T\right), \ \forall v \in \left(T_p \overline{\cM}^\ast _{k,\varepsilon}\right)^\perp\mbox{ with }|v| \leq \delta.
\end{equation*}
where $v^T=\left(\left\langle v, \nu_1^{k,\varepsilon}\right\rangle, \ldots,\left\langle v, \nu_{\bar{n}}^{k,\varepsilon}\right\rangle\right) \in \R^{\bar{n}}$ and $v^{\perp}=\left(\left\langle v, \varpi_1^{k,\varepsilon}\right\rangle, \ldots,\left\langle v, \varpi_l^{k,\varepsilon}\right\rangle\right) \in \R^l$.  Observe $\mathfrak{P}_{k,\varepsilon}$ are well defined for every $0<\varepsilon\le\varepsilon_0(k,\delta)$ due to the $C^2$-convergence of the mollified manifolds $\Sigma_{k,\varepsilon}$. Now, to see that $\psi_{k,\varepsilon}^\ast \to0$, in $C^{2,\kappa/2}$-topology, consider the map
\begin{equation}\label{Eq:ImplicitFunctionTheorem}
    \mathfrak{I}_{k,\varepsilon}: \overline{\cM}^\ast_k \times \R^{\bar{n}} \times \R^l \ni(p, z, w) \mapsto p+z^j \nu_j^{k,\varepsilon}+w^j \varpi_j^{k,\varepsilon}\in \R^{m+n},
\end{equation}
where we use the Einstein convention of summation over repeated indices. It is simple to show that the frame can be chosen so that $D \mathfrak{I}_{k,\varepsilon}(0,0)=\mathrm{Id}$ and, hence, using the implicit function theorem, $\mathfrak{I}_{k,\varepsilon}^{-1}\left(\overline{\Sigma}_{k,\varepsilon}\right)$ can be written locally as a graph of a function $\psi^\ast_{k,\varepsilon}$ that meets the claimed property. By construction, we also have $\psi^\ast_{k,\varepsilon}(x, 0)=0$, (which yields $D_x \psi^\ast_{k,\varepsilon}(x, 0)=0$) and $\left|D_z \psi^\ast_{k,\varepsilon}(x, 0)\right|=0$ for every $x \in \overline{\cM}^\ast _{k,\varepsilon}$, which in turn implies
\begin{equation}\label{Eq:C2AlphaEstimates(7.5)1}
    \left|D_x \psi^\ast_{k,\varepsilon}(x, z)\right| \leq C|z|^{1+\beta_1}, \ \quad\left|D_z\psi^\ast_{k,\varepsilon}(x, z)\right| \leq C|z|, \mbox{ and } \left|\psi^\ast_{k,\varepsilon}(x, z)\right| \leq C|z|^{2},
\end{equation}
where $C>0$ is a constant that can be chosen independently of $\varepsilon$ and $k$. Indeed, it suffices to remember that $\overline{\Sigma}_{k,\varepsilon}\to\overline{\Sigma}_k$ in the $C^{2,\beta}$-topology and also that $\overline{\Sigma}_k$ converges to $\R^{m+\bar{n}}\times\{0\}$ to deduce the existence of the constant $C>0$ independent of of $\varepsilon$ and $k$ such that $\|\psi^\ast_{k,\varepsilon}\|_{C^{2,\beta_1}}\le C$. We define $\psi^\ast_k$ analogously substituting in \eqref{Eq:ImplicitFunctionTheorem} $\overline{\Sigma}_{k,\varepsilon}$ by $\overline{\Sigma}_k$ and $\overline{\cM}^\ast _{k,\varepsilon}$ by $\overline{\cM}^\ast_k$. Notice that $\psi^\ast_k$ is of class $C^{1,\kappa}$ and that the convergence of $\psi^\ast_{k,\varepsilon}$ to $\psi^\ast_k$ is in the $C^{1,\kappa/2}$-topology, which is enough for the approximation procedure that we will perform later. 

Now given any $Q$-valued map $u_{k,\varepsilon}=\sum_i \a{ u_{k,\varepsilon,i}}: \overline{\cM}^\ast _{k,\varepsilon} \rightarrow\cA_Q\left(\{0\}\times\R^{\bar{n}} \times\{0\}\right)$ with $\|u_{k,\varepsilon}\|_{L^{\infty}} \leq \delta$, we can consider the map $\mathbf{u}_{k,\varepsilon}$ from $\overline{\cM}^\ast _{k,\varepsilon}$ to $\cA_Q\left(\nu(\overline{\cM}^\ast _{k,\varepsilon})\right)$, where $\nu(\overline{\cM}^\ast _{k,\varepsilon})$ is the total space of the normal bundle of $\overline{\cM}^\ast _{k,\varepsilon}$, defined by
\begin{equation*}
    \mathbf{u}_{k,\varepsilon}(x) :=  \sum_i \a{\sum_{j=1}^{\bar{n}}\left(u_{k,\varepsilon,i}\right)^j(x) \nu_j^{k,\varepsilon}(x)+\sum_{j=1}^{l}\psi^{\ast,j}_{k,\varepsilon}\left(x, u_{k,\varepsilon,i}(x)\right) \varpi_j^{k,\varepsilon}(x)},    
\end{equation*}
where we set $\left(u_{k,\varepsilon,i}\right)^j(x):=\left\langle u_{k,\varepsilon,i}(x), e_{m+j}\right\rangle, \psi^{\ast,j}_{k,\varepsilon}\left(x, u_{k,\varepsilon,i}(x)\right):=\left\langle\psi^\ast_{k,\varepsilon}\left(x, u_{k,\varepsilon,i}(x)\right), e_{m+\bar{n}+j}\right\rangle$. Then, we get that
\begin{equation*}
\begin{aligned}
    D\left(\mathbf{u}_{k,\varepsilon}\right)_i=D\left(u_i\right)^j \nu_j^{k,\varepsilon} &+ \left[D_x \psi^{\ast,j}_{k,\varepsilon}\left(x, u_i\right)+D_z \psi^{\ast,j}_{k,\varepsilon}\left(x, u_i\right) D u_i\right] \varpi_j^{k,\varepsilon}\\
    &+\left(u_i\right)^j D \nu_j^{k,\varepsilon}+\psi^{\ast,j}_{k,\varepsilon}\left(x, u_i\right) D \varpi_j^{k,\varepsilon},\text{ a.e. } \cH^m\res\overline{\cM}^\ast_{k,\varepsilon},    
\end{aligned}
\end{equation*}
where we used the Einstein summation convention on the index $j$. We define $\mathbf{u}_k$ analogously and also derive the analogous of the last displayed inequality for $\mathbf{u}_k$. Taking into account that 
\begin{equation}
    \lim_{k\to +\infty}\left\|D \nu_i^{k}\right\|_{C^0}+\left\|D \varpi_j^{k}\right\|_{C^0} = 0\mbox{ and }\lim_{\varepsilon\to 0}\left\|D \nu_i^{k,\varepsilon}-D\nu_i^k\right\|_{C^0}+\left\|D \varpi_j^{k,\varepsilon}-D \varpi_j^{k}\right\|_{C^0},
\end{equation}
we obtain the existence of a double sequence $(\varepsilon_0(k,\delta))_{k\in\mathbb{N}, \delta\in (0,+\infty)}$, generating a sequence $0<\varepsilon_{0,k} =\varepsilon_0(k,\delta_k) >0 $ such that $\varepsilon_{0,k}\to0^+$ and $\delta_k\to0^+$ as $k\to+\infty$. Furthermore, we have the following
\begin{equation*}
    \forall \{\varepsilon_k\}_{k\in\N}\mbox{ with }0<\varepsilon_k\le\varepsilon_{0,k},\forall k\in\N, \mbox{ we have } \lim_{k\to +\infty}\left\|D \nu_i^{k,\varepsilon_k}\right\|_{C^0}+\left\|D \varpi_j^{k,\varepsilon_k}\right\|_{C^0} = 0.
\end{equation*}
By \eqref{Eq:C2AlphaEstimates(7.5)1}, we readily derive that, for some constant $C>0$ independent of $k$ and $k$ large enough, it holds
\begin{equation*}
    \left|\int\left(\left|D \mathbf{u}_{k,\varepsilon_k}\right|^2-|D u|^2\right)\right| \leq C \int |u|^{2+2 \beta_1} + |u|^2|Du|^2  +o(1) |u|^2.
\end{equation*}
On the other hand, since $\psi^\ast_{k,\varepsilon}\to\psi^\ast_k$ in $C^{1,\kappa/2}$-topology as $\varepsilon\to0$ for any fixed $k$ large enough, we have
\begin{equation*}
    \lim_{\varepsilon\to 0}\left|\int\left(\left|D \mathbf{u}_{k,\varepsilon}\right|^2-|D\mathbf{u}_k|^2\right)\right|=0.
\end{equation*}
We now choose $\varepsilon_{k,\delta}>0$, depending on $\delta>0$ and $k\in\mathbb{N}$ large enough, to get from the last two displayed inequalities that 
\begin{equation}\label{Eq:7.6}
\begin{aligned}
    \left|\int\left(\left|D \mathbf{u}_k\right|^2-|D u|^2\right)\right| & \leq \left|\int\left(\left|D \mathbf{u}_{k,\varepsilon_{k,\delta}}\right|^2-|D u|^2\right)\right|+\left|\int\left(\left|D \mathbf{u}_{k,\varepsilon_{k,\delta}}\right|^2-|D\mathbf{u}_k|^2\right)\right|\\ 
    & \leq C \int |u|^{2+2 \beta_1} + |u|^2|Du|^2  +o(1) |u|^2.
\end{aligned}
\end{equation}
By \eqref{E:T:NA:LipC0}, we can write  ${\overline{N}_k^\ast}(x)=Q\a{\mathfrak{P}_{k}(x)} \oplus \psi^\ast_k\left(x, \bar{u}_k(x)\right)$ for some Lipschitz $\bar{u}_k:=\sum_i \a{\bar{u}_{k,i}}: \overline{\cM}^\ast_k \rightarrow$ $\cA_Q\left(\{0\}\times\R^{\bar{n}}\times\{0\}\right)$ with $\left\|\bar{u}_k\right\|_{L^{\infty}}=o(1)$. Setting 
\begin{equation*}
    u_k^b(x):=\sum_i \a{\mathfrak{P}_{k}(\mathbf{e}_k(x))+\bar{u}_{k,i}(\mathbf{e}_k(x))},
\end{equation*}
we conclude, from the Reverse Sobolev inequality (\cite[(5.11)]{DS5}), \eqref{E:T:NA:LipC0}, and \eqref{Eq:7.6}, that
\begin{equation}
\label{Eq:(7.7)TheLimitFunctionBelongsToW12}
    \lim _{k \rightarrow+\infty} \int_{B_{3/2}}\left(\left|D N_k^{\ast,b}\right|^2-\bh_k^{-2}\left|D u_k^b\right|^2\right)=0,
\end{equation}
and $N_{\infty}^{\ast,b}$ is the limit of $\bh_k^{-1} u_k^b$ in $ W^{1,2}(B_{3/2}, \cA_Q(\{0\}\times \R^{\bar{n}} \times \{0\}))$. 

\emph{\underline{Step 3}: Proof of the $\operatorname{Dir}$-minimizing property of $N_{\infty}^{\ast b}$ ((i) of \cref{T:FinalBU})}

There is nothing to prove if its Dirichlet energy vanishes. Therefore, we assume existence of $c_0>0$ such that $0<c_0 {\bh_k^\ast}^2 \leq \int_{\mathcal{B}^\ast_{3/2}}\left|D {\overline{N}_k^\ast}\right|^2$. We will argue by contradiction, if $(i)$ of \cref{T:FinalBU} were to be false, we argue follow the same argument in \cite[Section 7.3]{DS5} to find $r\in (t,2)$ and $v_k^{\ast,b}$ with the following property. If we define $\widetilde{N}_{k,\varepsilon}^\ast = \psi^\ast_{k,\varepsilon}\left(x, v_k^b \circ \mathbf{e}_{k,\varepsilon}^{-1}\right)$, then we have
\begin{equation*}
\begin{aligned}
    & \widetilde{N}_{k,\varepsilon}^\ast \equiv \overline{N}^\ast_{k,\varepsilon} \quad \text { on } \mathcal{B}_{3/2} \backslash \mathcal{B}_t, \quad \operatorname{Lip}\left(\widetilde{N}_{k,\varepsilon}^\ast\right) \leq C {\bh_{k,\varepsilon}^\ast}^{\gamma_{na}}, \quad\left|\widetilde{N}_{k,\varepsilon}^\ast\right| \leq C\left({\bm_{0,k}^\ast} \bar{r}_k\right)^{\gamma_{na}}\\
    & \int_{\mathcal{B}_{\frac{3}{2}}}\left|\boldsymbol{\eta} \circ \widetilde{N}_{k,\varepsilon}^\ast\right| \leq C {\bh_{k,\varepsilon}^\ast}^2 \quad \text { and } \quad \int_{\mathcal{B}_{3/2}}\left|D \widetilde{N}_{k,\varepsilon}^\ast\right|^2 \leq \int_{\mathcal{B}_{3/2}}\left|D {\overline{N}_k^\ast}\right|^2-\delta {\bh_{k,\varepsilon}^\ast}^2 .
\end{aligned} 
\end{equation*}
We finally construct the competitor current that will violate the minimality of $T$. Set $\widetilde{F}^\ast_{k,\varepsilon}(x)=\sum_i \a{ x+\widetilde{N}_{k,\varepsilon,i}^\ast(x) }$. The currents $\mathbf{T}_{\widetilde{F}^\ast_{k,\varepsilon}}$ coincides with $\mathbf{T}_{\bar{F}^\ast_{k,\varepsilon}}$ in ${(\bp^\ast_{k,\varepsilon})}^{-1}\left(\mathcal{B}_{3/2} \backslash \mathcal{B}_t\right)$ and both of them lie in $\overline{\Sigma}_{k,\varepsilon}$. Define the function $\varphi^\ast_{k,\varepsilon}(p)=\operatorname{dist}_{\overline{\cM}^\ast _{k,\varepsilon}}\left(0, \bp^\ast_{k,\varepsilon}(p)\right)$, and, for each $s \in (t, \frac{3}{2})$, consider the slices $\left\langle\mathbf{T}_{\widetilde{F}_{k,\varepsilon}}-\bar{T}_{k,\varepsilon}, \varphi_{k,\varepsilon}, s\right\rangle$, where $\bar{T}_{k,\varepsilon}:=\left(\mathbf{\mathcal{T}}_{k,\varepsilon}\right)_{\sharp}\bar{T}_k$ and $\mathbf{\mathcal{T}}_{k,\varepsilon}:\overline{\Sigma}_k\to\overline{\Sigma}_{k,\varepsilon}$ are the natural $C^{2,\beta}$ diffeomorphisms induced by the mollification process which by definition converge to $\mathbf{1}_{\overline{\Sigma}_k}$ in $C^{2,\beta}$ topology. By \eqref{Eq:(7.1)}, we have
\begin{equation*}
    \int_t^{\frac{3}{2}} \mathbf{M}\left(\left\langle\mathbf{T}_{\widetilde{F}^\ast_{k,\varepsilon}}-\bar{T}_{k,\varepsilon}, \varphi^\ast_{k,\varepsilon}, s\right\rangle\right) \leq C {\bh_{k,\varepsilon}^\ast}^{2+\gamma_{na}} .    
\end{equation*}
Thus, for each $k\in\mathbb{N}$ and $\varepsilon>0$, we can find a radius $\sigma_{k,\varepsilon} \in ( t, \frac{3}{2} )$ for which 
\begin{equation*}
     \mathbf{M}\left(\left\langle\mathbf{T}_{\widetilde{F}^\ast_{k,\varepsilon}}-\bar{T}_{k,\varepsilon}, \varphi^\ast_{k,\varepsilon}, \sigma _{k,\varepsilon}\right\rangle\right) \leq C {\bh_{k,\varepsilon}^\ast}^{2+\gamma_{na}} .    
\end{equation*}
By the isoperimetric inequality (see \cite[Rem. 4.3]{DS3}), there is a current $S_{k,\varepsilon}$ such that
\begin{equation*}
    \partial S_{k,\varepsilon}=\left\langle\mathbf{T}_{\widetilde{F}^\ast_{k,\varepsilon}}-\bar{T}_{k,\varepsilon}, \varphi^\ast_{k,\varepsilon}, \sigma_{k,\varepsilon}\right\rangle, \quad \mathbf{M}\left(S_{k,\varepsilon}\right) \leq C {\bh^\ast_{k,\varepsilon}}^{(2+\gamma) m /(m-1)} \quad \text { and } \operatorname{spt}\left(S_{k,\varepsilon}\right) \subset \overline{\Sigma}_{k,\varepsilon}.
\end{equation*}
Our competitor current is then given by
\begin{equation*}
    \widetilde{T}_{k,\varepsilon}:=\bar{T}_{k,\varepsilon}\res\left((\bp_{k,\varepsilon}^\ast)^{-1}\left(\overline{\cM}^\ast _{k,\varepsilon} \backslash \mathcal{B}_{\sigma_{k,\varepsilon}}\right)\right)+S_{k,\varepsilon}+\mathbf{T}_{\widetilde{F}^\ast_{k,\varepsilon}}\res\left((\bp_{k,\varepsilon}^\ast)^{-1}\left(\mathcal{B}_{\sigma_{k,\varepsilon}}\right)\right) .    
\end{equation*}
Note that $\widetilde{T}_{k,\varepsilon}$ is supported in $\overline{\Sigma}_{k,\varepsilon}$. On the other hand, by \eqref{Eq:(7.1)} and the bound on $\mathbf{M}\left(S_{k,\varepsilon}\right)$, we have
\begin{equation*}
    \mathbf{M}\left(\widetilde{T}_{k,\varepsilon}\right)-\mathbf{M}\left(\bar{T}_{k,\varepsilon}\right) \leq \mathbf{M}\left(\mathbf{T}_{\overline{F}^\ast_{k,\varepsilon}}\right)-\mathbf{M}\left(\mathbf{T}_{\widetilde{F}^\ast_{k,\varepsilon}}\right)+C \left(\bh^\ast_{k,\varepsilon}\right)^{2+2 \gamma_{na}} .    
\end{equation*}
Denote by $\overline{A}^\ast _{k,\varepsilon}$ and $\overline{H}^\ast _{k,\varepsilon}$ the second fundamental forms and mean curvatures of the manifold $\overline{\cM}^\ast _{k,\varepsilon}$, respectively. Using the above inequality and the Taylor expansion of \cite[Th. 3.2]{DS2}, for every $0<\varepsilon\le\varepsilon_{0,k}$, we achieve
\begin{equation}\label{Eq:(7.10)}
\begin{aligned}
    \mathbf{M}\left(\widetilde{T}_{k,\varepsilon}\right)-\mathbf{M}\left(\bar{T}_{k,\varepsilon}\right)\leq & \frac{1}{2} \int\left(\left|D \widetilde{N}_{k,\varepsilon}^\ast\right|^2-\left|D \overline{N}^\ast_{k,\varepsilon}\right|^2\right) \\
    &\quad + C\left\|\overline{H}^\ast _{k,\varepsilon}\right\|_{C^0} \int\left(\left|\boldsymbol{\eta} \circ \overline{N}^\ast_{k,\varepsilon}\right|+\left|\boldsymbol{\eta} \circ \widetilde{N}_{k,\varepsilon}^\ast\right|\right) \\
    &\quad + \left\|\overline{A}^\ast _{k,\varepsilon}\right\|_{C^0}^2 \int\left(\left|\overline{N}^\ast_{k,\varepsilon}\right|^2+\left|\widetilde{N}_{k,\varepsilon}^\ast\right|^2\right)+o\left({\mathbf{h}^\ast_{k,\varepsilon}}^2\right)\\ 
    &\leq -\frac{\delta}{2} {\mathbf{h}^\ast_{k,\varepsilon}}^2+o\left({\mathbf{h}^\ast_{k,\varepsilon}}^2\right),
\end{aligned}
\end{equation}
where in the last inequality we take into account \cite[Lemma 6.1]{DS5}, which can be used here in this form because, by construction, our $\overline{\cM}^\ast_{k,\varepsilon}$ center manifolds are all of class $C^\infty$ and so in particular of class $C^{3,\beta_1}$ having also $C^2$ norm uniformly bounded with respect to $k$ and $\varepsilon$. 

To finish the proof, we have to take the limit on $\varepsilon$ in \eqref{Eq:(7.10)}. We argue as follows. For every one parameter family pair of abstract pointed Riemannian manifolds $\left(\left(\Sigma_\varepsilon, g_\varepsilon, p_\varepsilon\right),\left( \Sigma^\prime_\varepsilon, g^\prime_\varepsilon, p^\prime_\varepsilon\right)\right)_\varepsilon$ whose underlying differentiable structure is of class $C^1$ and with metric tensors $g_\varepsilon,g^\prime_\varepsilon$ of class $C^0$ converging in topology $C^0$ to the pair $\left(\left(\Sigma_0, g_0, p_0\right),\left( \Sigma^\prime_0, g^\prime_0, p^\prime_0\right)\right)$, and every family of $C^1$ diffeomorphisms $\mathfrak{D}_\varepsilon\in C^1\left(\Sigma_\varepsilon,\Sigma^\prime_\varepsilon\right)$ and $\mathfrak{D}_\varepsilon^{-1}\in C^1\left(\Sigma^\prime_\varepsilon,\Sigma_\varepsilon\right)$, it is easy to check that $\left(\mathfrak{D}_\varepsilon\right)_{\sharp}(\bI_m(\Sigma_\varepsilon,g_\varepsilon))=\bI_m(\Sigma^\prime_\varepsilon,g^\prime_\varepsilon)$ (here $\bI_m(\Sigma_\varepsilon,g_\varepsilon)$ denotes all integral $m$-currents in $(\Sigma_\varepsilon,g_\varepsilon)$). Moreover, if we assume that there exists an isometry of pointed metric spaces $\mathbf{\mathfrak{f}}:(\Sigma_0, g_0, p_0)\to(\Sigma^\prime_0,g^\prime_0,p^\prime_0)$ with $\mathbf{\mathfrak{f}}\in C^1\left(\Sigma,\Sigma^\prime\right)$ and $\mathbf{\mathfrak{f}}^{-1}\in C^1\left(\Sigma^\prime,\Sigma\right)$ satisfying $\mathfrak{D}_\varepsilon\to\mathbf{\mathfrak{f}}$ in $C^1$ topology, $T_\varepsilon\in\bI_m(\Sigma_\varepsilon,g_\varepsilon)$, $T_\varepsilon\to T_0\in \bI_m(\Sigma_0, g_0)$ in the intrinsic flat topology, then it holds
\begin{equation}\label{Eq:MassConvergenceOfMollifiedCurrents}
    \mathbf{M}_{(\Sigma^\prime_\varepsilon,g^\prime_\varepsilon)}\left(\left(\mathfrak{D}_\varepsilon\right)_{\sharp}T_\varepsilon\right)\to\mathbf{M}_{(\Sigma^\prime_0,g^\prime_0)}\left(\mathbf{\mathfrak{f}}_{\sharp}T_0\right)=\mathbf{M}_{(\Sigma_0,g_0)}\left(T_0\right), \mbox{ when }\varepsilon\to 0^+.   
\end{equation}
Applying \eqref{Eq:MassConvergenceOfMollifiedCurrents} to our setting immediately provides the following
\begin{equation*}
    \mathbf{M}\left(\widetilde{T}_k\right)-\mathbf{M}\left(\bar{T}_k\right)= \lim_{\varepsilon\to0^+}\mathbf{M}\left(\widetilde{T}_{k,\varepsilon}\right)-\mathbf{M}\left(\bar{T}_{k,\varepsilon}\right),
\end{equation*}
and, by Dominated convergence theorem, we have that ${\mathbf{h}^\ast_{k,\varepsilon}}\to{\bh_k^\ast}$ as $\varepsilon\to 0^+$ for each fixed $k$ large enough. This together with \eqref{Eq:(7.10)} contradicts the minimizing property of $\bar{T}_k$ for $k$ large enough and $\varepsilon_{0,k}$ small enough. Hence, we conclude the proof.
\end{proof}

            \bibliographystyle{abbrv}
            \addcontentsline{toc}{section}{References} 
            \bibliography{main}
\end{document}